\documentclass{article}
\usepackage{subfigure}  
\usepackage{amsmath,amsthm,amssymb,graphicx}
\usepackage{multirow}
\usepackage{rotating}
\usepackage{sidecap}
\usepackage{mathtools}

\usepackage{marvosym}
\usepackage{color}
\newcommand{\N}{\mathbb{N}}
\newcommand{\R}{\mathbb{R}}
\usepackage[labelformat=empty]{caption}

\DeclareMathOperator*{\Span}{span}
\DeclareMathOperator*{\dom}{dom}

\newcommand{\C}{\mathbb{C}}

\numberwithin{equation}{section}

\newcounter{dummy} \numberwithin{dummy}{section}

\newtheorem{defn}[dummy]{Definition}
\newtheorem{rmk}[dummy]{Remark}

\newtheorem{thm}{Theorem}
\newtheorem{cor}[dummy]{Corollary}
\newtheorem{lem}[dummy]{Lemma}
\newtheorem{prop}{Proposition}

\newcommand{\thmref}[1]{Theorem~\ref{#1}}
\newcommand{\lemref}[1]{Lemma~\ref{#1}}
\newcommand{\propref}[1]{Proposition~\ref{#1}}
\newcommand{\corref}[1]{Corollary~\ref{#1}}
\newcommand{\rmkref}[1]{Remark~\ref{#1}}

\usepackage{xspace}
\usepackage{graphicx}
\usepackage{pdfsync}
\usepackage{amsmath}
 \usepackage{amssymb}
 \usepackage{amsthm}
\usepackage{fncylab}

\usepackage{slashed}

\newcommand{\Ec}{\mathcal{E}}
\newcommand{\nablaT}{\widetilde{\nabla}}
\newcommand{\KH}{K_{H}}

\newcommand{\Tw}{T_{w}}

\newcommand{\lip}{\text{Lip}}

\usepackage{lipsum}

\newcommand\blfootnote[1]{%
  \begingroup
  \renewcommand\thefootnote{}\footnote{#1}%
  \addtocounter{footnote}{-1}%
  \endgroup
}

\begin{document}

\title{\textsc{dirac operators and geodesic metric on the harmonic sierpinski gasket and other fractal sets}}

\author{Michel L. Lapidus$^1$ \ and \ Jonathan J. Sarhad$^2$}

\maketitle

\begin{abstract}

 We construct Dirac operators and spectral triples for certain, not necessarily self-similar, fractal sets built on curves.  Connes' distance formula of noncommutative geometry provides a natural metric on the fractal.  To motivate the construction, we address Kigami's measurable Riemannian geometry, which is a metric realization of the Sierpinski gasket as a self-affine space with continuously differentiable geodesics.  As a fractal analog of Connes' theorem for a compact Riemmanian manifold, it is proved that the natural metric coincides with Kigami's geodesic metric.  This present work extends to the harmonic gasket and other fractals built on curves a significant part of the earlier results of E. Christensen, C. Ivan, and the first author obtained, in particular, for the Euclidean Sierpinski gasket.  (As is now well known, the harmonic gasket, unlike the Euclidean gasket, is ideally suited to analysis on fractals.  It can be viewed as the Euclidean gasket in harmonic coordinates.)  Our current, broader framework allows for a variety of potential applications to geometric analysis on fractal manifolds.

\end{abstract}


\blfootnote{2010 \emph{Mathematics Subject Classification}. \emph{Primary} 28A80,
 34L40, 46L87, 53C22, 53C23, 58B20, 58B34. \emph{Secondary} 53B21,
 53C27, 58C35, 58C40, 81R60.}

\blfootnote{\emph{Key words and phrases}. Analysis on fractals, noncommutative
 fractal geometry, Laplacians and Dirac operators on fractals,
 spectral triples, spectral dimension, measurable Riemannian geometry, geodesics on
 fractals, geodesic and noncommutative metrics, fractals built on
 curves, Euclidean and harmonic Sierpinski gaskets,
 geometric analysis on fractals, fractal manifold.}

\blfootnote{The research of M. L. Lapidus was partially supported by the National Science Foundation under grant nos.  DMS-0707524 and DMS-110755, as well as by the Institut des Hautes Etudes Scientifiques (IHES) where he was a visiting professor in the spring of 2012 while this paper was completed.\\ 

 \hspace{-4mm} $^1$Department of Mathematics, University of California, Riverside, CA 92521-0135, USA 
 
\hspace{-2.5mm} E-mail address:  lapidus@math.ucr.edu \\ 

\noindent $^2$Department of Biology, University of California, Riverside, CA 92521, USA 

 \hspace{-2.5mm} E-mail address:  jonathan.sarhad@ucr.edu}

\section{Introduction}

In this article, we provide a general construction of a Dirac operator and its associated spectral triple for a large class of sets built on curves, which includes the self-similar Sierpinski gasket, the self-affine harmonic gasket, and other spaces which carry an intrinsic metric.  Using methods from noncommutative geometry, it is shown that the intrinsic metric can be recovered from the spectral triple.  In this sense, there are `target' geometries in mind, which are recovered using the operator-theoretic data contained in a spectral triple.  Informally, this method involves using a space of functions on the underlying space as coordinates.  If the function space is a commutative $C^{*}$-algebra, then one can tease from it a topological space.  This is a consequence of Gelfand's theorem.  If that topological space is metrizable, then more information is needed in order to determine a metric.  Knowledge of a certain Hilbert space of vector fields on the space and a particular differential operator is enough to determine a metric in many instances. This way of constructing a geometry is part of the broader theory of noncommutative geometry. \\

Alain Connes \cite{Con1, Con2} proved that for a compact spin Riemannian manifold, $M$, a triple of objects, called a spectral triple, encodes the geometry of $M$.  The spectral triple consists of the $C^{*}$-algebra of complex-continuous functions on $M$, the Hilbert space of $L^{2}$-spinor fields, and a differential operator called the \textsl{Dirac operator}.  The Dirac operator is constructed from the spin connection associated to $M$ and can be thought of as the square root of the spin-Laplacian (mod scalar curvature).  Connes' formula, though very simple, uses the information from the spectral triple in order to recover the geodesic distance on $M$, and hence the geometry of $M$ (by the Meyers--Steenrod Theorem \cite{Peterson}).  The observation that the Dirac operator defines the geometry is one of Connes' contributions to the field of geometry \cite{Var}.  Indeed, it is a springboard for defining generalized manifolds and geometries in the context of spaces which admit a meaningful generalization of the Dirac operator, but not meaningful generalizations of smooth structure or metric or even paths in the space.  In the absence of spin or even orientability, this result still holds, though the Dirac operator may not be uniquely defined.  The reason for the name noncommutative geometry is that the arguments involved in this result do not rely on the commutativity of the $C^{*}$-algebra, which opens the door to the possibility of defining geometries on noncommutative $C^{*}$-algebras.  The applications of noncommutative geometry in this article, however, stay within the context of the commutative $C^{*}$-algebras of complex-continuous functions on a class of sets. \\

Previous work by Michel Lapidus has provided applications of the methods of noncommutative geometry to fractals.  His program for viewing fractals as generalized manifolds and possibly noncommutative spaces is outlined in \cite{Lap4}.   Building in particular upon \cite{Lap3} and \cite{Lap2}, he investigated in many different ways the possibility of developing a kind of noncommutative fractal geometry, which would merge aspects of analysis on fractals (as now presented, e.g., in \cite{Kig3}) and Connes' noncommutative geometry \cite{Con2}. (See also parts of \cite{Lap2} and \cite{Lap3}.) Central to \cite{Lap4} was the proposal to construct suitable spectral triples that would capture the geometric and spectral aspects of a given self-similar fractal, including its metric structure.  In \cite{CrIv1}, Erik Christensen and Cristina Ivan constructed a spectral triple for the approximately finite-dimensional (AF) $C^*$-algebras. The continuous functions on the Cantor set form an AF $C^*$-algebra since the Cantor set is totally disconnected. Hence, it was quite natural to try to apply the general results of that paper for AF $C^*$-algebras to this well-known example. In this manner, they showed in \cite{CrIv1} once again how suitable noncommutative geometry may be to the study of the geometry of a fractal. Since then, the authors of the present article have searched for possible spectral triples associated to other known fractals. The hope is that these triples may be relevant to both fractal geometry and analysis on fractals. We have been especially interested in the Sierpinski gasket, a well-known nowhere differentiable planar curve, because of its key role in the development of harmonic analysis on fractals. (See, for example, \cite{Barlow}, \cite{Kig3}, \cite{KL1}, \cite{KL2}, \cite{Strichartz}, \cite{Tep1}.)  In \cite{CrIvLap}, Erik Christensen, Christina Ivan and Michel Lapidus applied these noncommutative methods to some fractal sets built on curves---including trees, graphs, and the Sierpinski gasket.  The work in \cite{CrIvLap} on the more complex sets is based largely on the Dirac operator and spectral triple on the circle.  It is important to note that the work in \cite{CrIvLap} on the Sierpinski gasket is with respect to the Sierpinski gasket in the Euclidean metric as opposed to the treatment of the Sierpinski gasket in the harmonic metric of the present paper.  Of many results in \cite{CrIvLap}, the application of noncommutative methods to the Euclidean Sierpinski gasket recovered the geodesic distance, volume measure (in that case, the natural Hausdorff measure), and metric \textsl{spectral} dimension (there, the Hausdorff dimension).\\
   
   The Sierpinski gasket is a fractal set which is not a smooth manifold nor even a topological manifold.  It is shown below in Figure 1 (of Section 2) as it is usually viewed, in the \textsl{Euclidean metric}.  The Sierpinski gasket has a natural metric structure induced by the Euclidean metric in $\R^{2}$, given by the existence of a shortest path (non unique) between any two points. These shortest paths are piecewise Euclidean segments and hence piecewise differentiable, but in general not differentiable.  In \cite{Kig4} (see also \cite{Kig5}), Jun Kigami uses a theory of harmonic functions on the Sierpinski gasket (see, e.g., \cite{Kig3}) in order to construct a new metric space that is homeomorphic to the Sierpinski gasket.  This new space shown below (Figure 2 of Section 2), called the \textsl{harmonic gasket} or the \textsl{Sierpinski gasket in harmonic coordinates}, is actually given by a single harmonic coordinate chart for the Sierpinski gasket.  The harmonic gasket has a $C^{1}$ shortest path (non unique) between any two points.  It is interesting to note that the harmonic coordinate chart smoothes out the Sierpinski gasket.  Kigami \cite{Kig1, Kig4, Kig5}, building on work by Kusuoka \cite{Kus1, Kus2}, has found several formulas in the setting of the harmonic gasket which are measurable analogs to their counterparts in Riemannian geometry.  In particular, he has found formulas for energy and geodesic distance involving measurable analogs to Riemannian metric, Riemannian gradient, and Riemannian volume.  For this reason, this geometry is appropriately called measurable Riemannian geometry.\\ 

In this article, as an example, we recover Kigami's measurable Riemannian geometry using spectral triples.  As in \cite{CrIvLap}, here the basis for the construction of these spectral triples is the spectral triple for a circle:  Finite and countable direct sums of the circle triples are used to construct the desired spectral triples.  We have constructed several spectral triples for the harmonic gasket, all of which recover the geodesic distance on the harmonic gasket as well as on the (Euclidean) Sierpinski gasket.  The general construction provided in \propref{mainprop} below applies to a class of sets built on countable unions of curves in $\R^{n}$ which includes the Sierpinski gasket and harmonic gasket.  The spectral triple on the harmonic gasket provides a \textsl{fractal} analog to Connes' theorem.  Indeed, on the one hand, there is a \textsl{target} geometry that is a fractal analog of Riemannian geometry---Kigami's measurable Riemannian geometry.  On the other hand, there is our construction of the Dirac operator and spectral triple for fractal sets which can be used to recover Kigami's geometry, namely through the Dirac operator.\\

We point out that our results, which make use of and extend the methods of \cite{CrIvLap}, encompass the results of \cite{CrIvLap} concerning the construction of Dirac operators and the recovery of the geodesic metric.  Furthermore, our results allow more flexibility and are better suited to a further development of geometric analysis on fractals.  Indeed, in particular, in light of the results of \cite{Kus1, Kus2, Kig4, Kig5, Tep1, Tep2, Kaj2, Kaj3}, the harmonic gasket (rather than the ordinary Euclidean gasket) is the appropriate model for studying probability theory and harmonic analysis as well as the analog of Riemannian geometry on such a fractal.  Recent developments (some of which are alluded to in Section 7) suggest that many other fractal geometries can be similarly viewed as fractal (Riemannian) manifolds.\\

In the concluding remarks of this article, in addition to providing several additional references relevant to this paper, we discuss work in progress which includes a different construction of a Dirac operator and spectral triple from the ones built from direct sums.  This \textsl{global} Dirac operator is defined directly from Kigami's measurable Riemannian metric and gradient, giving it a stronger resemblance to Connes' Dirac operator on a compact Riemannian manifold.  The Hilbert space of the triple is  constructed from Kigami's $L^{2}$-vector fields on the gasket, again giving a stronger fractal analog to Connes' theorem.  In addition,  the global construction may prove a better starting point for showing that the Dirac operator squares to the appropriate Laplacian in this setting, the Kusuoka Laplacian.  We also discuss two open problems.  These problems, which are inherently linked, are the computation of the spectral dimension and volume measure induced by the spectral triples for the harmonic gasket.\\

The remainder of this paper is organized as follows:\\

In Section 2 are provided various preliminaries concerning spectral triples and Connes' formula, some of the methods and results of \cite{CrIvLap} which we will extend, as well as analysis on fractals (focusing on the Euclidean and harmonic Sierpinski gaskets) and the results of \cite{Kig4} concerning measurable Riemannian geometry, particularly the construction of the geodesic metric and the existence of $C^1$ (but not usually $C^2$) geodesics on the harmonic gasket.\\

In Section 3, we discuss spectral geometry on a new class of fractal sets built on curves.  [In short, they are compact metric length spaces (see Definition 3.1 of Section 3) satisfying two basic axioms.]  These ``fractals'' (which are not necessarily self-similar or even ``self-alike'') include both discrete structures (such as certain infinite trees, as considered in \cite{CrIvLap}) and continuous structures (such as the Euclidean and harmonic gaskets).  We construct Dirac operators and associated spectral triples on such fractals.  We also show that one can recover the natural geodesic metric on them.\\

In Section 4 and Section 5, respectively, we show that the Euclidean gasket and the harmonic gasket belong to the class of fractals introduced in Section 3.  In particular, we deduce from the results obtained in Sections 3 and 4 that the Euclidean geodesic metric on the Sierpinski gasket can be recovered from the spectral triple (as was already done in \cite{CrIvLap}).  Furthermore, we deduce from the results obtained in Sections 3 and 5 the new fact according to which the $C^1$ geodesic metric of \cite{Kig4} can be recovered from the spectral triple constructed in Section 5.\\

In Section 6, we provide several alternative constructions of spectral triples associated with the harmonic gasket and compare the corresponding eigenvalue spectra and spectral dimensions.  We also show that they all induce the same noncommutative metric, namely, the harmonic geodesic metric (just as in Section 5).\\

Finally, in Section 7, as was mentioned in more detail above, we discuss further work connected to various aspects of the present paper, as well as propose several open problems and directions for future research in the area of noncommutative fractal geometry \cite{Lap4} and geometric analysis on fractals.\\

\section{Preliminaries}

\subsection{Spectral triples, Dirac operators, and noncommutative geometry}

 From the perspective of noncommutative geometry, the geometric information of a space is encoded in a triple.  One part of the triple is a $C^{*}$-algebra.  Recall that a $C^{*}$-algebra is a Banach algebra with a conjugate linear involution $*$ satisfying:  $(xy)^{*}=y^{*}x^{*}$ and  $||x^{*}x||=||x^{*}||||x||=||x||^{2}$.
Some relevant examples of $C^{*}$-algebras are the complex numbers $\C$, the complex continuous functions $C(X)$ on a compact Hausdorff space $X$, and the bounded linear operators $B(H)$ on a Hilbert space $H$. \\
 
The Gelfand--Naimark Theorem \cite{Con2, Var} states that every unital commutative $C^{*}$-algebra $\mathcal A$ is $*$-isomorphic to $C(X)$, for some compact Hausdorff space $X$.  The space $X$ is unique, up to homeomorphism.  In fact, $X$ is determined as the set of all pure states (characters) of $\mathcal A$, with the weak $^{*}$-topology  assigned.  Note that if $X$ is a compact metric space, then the weak$^{*}$-topology on the set of pure states is metrizable.  A second, more general, result due to Gelfand and Naimark is that \textsl{any} $C^{*}$-algebra can be faithfully represented in $B(H)$, for some Hilbert space $H$.   \\
  
The Gelfand--Naimark Theorem yields a perspective for partitioning topologies (or geometries) roughly through the following correspondences (modulo Morita equivalence, see \cite{Con2, Var}):

\begin{enumerate}

\item \textsl{commutative topologies/geometries} $\Longleftrightarrow$ commutative $C^{*}$-algebras;

\item \textsl{noncommutative topologies/geometries} $\Longleftrightarrow$ noncommutative $C^{*}$-algebras.

\end{enumerate}

Since $C(X)$ is commutative, we say that $X$ has a commutative topology or geometry.  In this way, one may consider noncommutative 
\textsl{rings of functions} on some `noncommutative spaces'.  The geometries presented in this article are examples of commutative, yet non-classical, geometries.\\

  Specifying a natural or intrinsic distance function on a set or space is central to noncommutative geometry.  In the context of $C^{*}$-algebras, it was first suggested by Connes (\cite{Con1}, see also \cite{Con2}) that from a suitable Lipschitz seminorm one obtains an ordinary metric on the state space of the $C^{*}$-algebra. (See also Reiffel's work in \cite{Rief1, Rief2} and the references therein for further abstractions and extensions of this point of view.)  Let $X$ be a compact metric space with metric $\rho$.  Defined on real-valued or complex-valued functions on $X$, the Lipschitz seminorm, $\lip_{\rho}$,  determined by $\rho$, is given by 
  
\begin{equation} \label{lip1}
\lip_{\rho}(f)=\sup\left\{\frac{|f(x)-f(y)|}{\rho(x,y)} : x\neq y\right\}. 
\end{equation}
  
\noindent The space of \textsl{$\rho$-Lipschitz functions on $X$} is comprised of those functions $f$ on $X$ satisfying $\lip_{\rho}(f)<\infty$.   One can recover the metric $\rho$, in a simple way, from $L_{\rho}$, by the following formula \cite{Rief2}:

\[ \rho(x,y)=\sup\{|f(x)-f(y)|: \lip_{\rho}(f)\leq 1\}.\]

In noncommutative geometry, a standard way to specify the \textsl{suitable} Lipschitz seminorm is via a \textsl{Dirac} operator $D$ on a Hilbert space $H$, the remaining parts of the triple.  Dirac operators have origin in quantum mechanics, but will be defined here in the context of \textsl{unbounded Fredholm modules} and \textsl{spectral triples}.  Following \cite{CrIvLap}, we will use the following definitions (see also, e.g., \cite{ConMarc}):

\begin{defn} let $\mathcal A$ be a unital $C^{*}$-algebra.  An \textbf{unbounded Fredholm module} $(H,D)$ over $\mathcal A$ consists of a Hilbert space $H$ which carries a unital representation $\pi$ of $\mathcal A$ and an unbounded, self-adjoint operator $D$ on $H$ such that
\pagebreak
\begin{enumerate}
\item[i.] the set 
\[\{a\in \mathcal A : [D,\pi(a)]\mbox{ is densely defined \& extends to a bounded operator on }H \}\] 
 is a dense subset of $\mathcal A$,
\item[ii.] the operator $(I+D^{2})^{-1}$ is compact.
\end{enumerate}
\end{defn} 

\begin{defn} Let $\mathcal A$ be a unital $C^{*}$-algebra and $(H,D)$ an unbounded Fredholm module of $\mathcal A$.  If the underlying representation $\pi$ is faithful, then $(\mathcal A, H,D)$ is called a \textbf{spectral triple}.  In addition, $D$ is called a \textsl{Dirac} operator.

\end{defn}

We will denote an unbounded Fredholm module $(H,D)$ over $\mathcal A$ as the triple $(\mathcal A, H, D)$ and call $D$ a Dirac operator, whether or not $\pi$ is faithful.\\

Using the information contained in the spectral triple for a compact spin Riemannian manifold $(M,g)$, Connes' Formula 1 below recovers the geodesic distance and hence the geometry of $(M,g)$.  Let $\mathcal A=C(M)$, $H$ be the Hilbert space of $L^{2}$-spinors, $D$ the Dirac operator associated to the spin connection on $(M,g)$, and let $d_{g}$ be the geodesic distance on $(M,g)$.  Connes' formula can now be stated (\cite{Con1}; \cite{Con2}, p. 544) as follows:

\newtheorem*{CO1}{Formula 1} 

\begin{CO1} For any points $p,q\in M$, we have

\[d_{g}(p,q)=\sup_{a\in \mathcal A}\{|a(p)-a(q)| : ||[D,\pi a]||\leq 1 \} ,\]

\end{CO1}

\noindent where $||.||$ denotes the norm on the space of bounded linear operators on $H$.  \\

We will usually refer to Formula 1 as the spectral distance or the distance induced by the spectral triple via Formula 1.  In all of our applications, $\pi$ is a representation as a multiplication operator and it will be clear that our Dirac operator $D$ satisfies

\[[D,\pi a](g)=\pi Da(g)=(Da)g .\]

\noindent In other words, the commutator operator is multiplication by the function $Da$.  Since the operator norm of a multiplication operator is equal to the essential supremum of the function by which it multiplies, we have

\[ ||[D,\pi a]||=||\pi Da||=||Da||_{\infty, M},\]

\noindent where in general $M$ will be a compact length space in $\R^{n}$.  This allows us to equivalently write the spectral distance as

 \[d_{g}(p,q)=\sup_{a\in \mathcal A}\{|a(p)-a(q)| : ||Da||_{\infty, M}\leq 1 \}. \]
 
\noindent Let $d$ be the metric on $M$ and $||.||_{\infty,M}$ denote the supremum norm on $M$.  Then $d_{g}=d$ if (and only if) $||Da||_{\infty,M}=\lip_{d}(a)$, where $\lip_{d}$ is the Lipschitz seminorm with respect to $d$ (see Equation \eqref{lip1}).  The brief argument for the `if' part of the statement is well known and is given in the proof of \thmref{mainthm} below.  Due to this relationship, several lemmas to follow show, for various settings, that $||Da||_{\infty,M}=\lip_{d}(a)$.  These lemmas allow us to recover the metric $d$ on $M$ as the spectral distance.\\

In \cite{CrIvLap}, an additional definition associated to a spectral triple is used to define the \textsl{spectral dimension} of the spectral triple.  (It is also called the \textsl{metric dimension} in \cite{ConMarc}.) This is a generalization of the dimension of a manifold---and indeed, in the case of a compact Riemannian manifold, recovers the dimension of the manifold \cite{Con2}.  (See also, for example, \cite{ConSull, Lap3, Lap4, KL2, GuidIso1, GuidIso2, CrIv1, CrIv2, CrIvLap, CipSauv, CrIvSc, CipGuidIso1} for the case of fractal spaces.)  This information is contained in the pairing of the Dirac operator and the Hilbert space, in the form of the asymptotics of the eigenvalues of the Dirac operator:\\

\begin{defn}  Let $D$ be the Dirac operator associated to the spectral triple in Definitions 1 and 2.  If $Tr((I+D^{2})^{-p/2})<\infty$ for some positive real number $p$, then the spectral triple is called \textbf{p-summable} or just \textbf{finitely summable}.  The number $\partial_{ST}$, given by

\[\partial_{ST}=\inf\{p>0: tr(D^{2}+I)^{\frac{-p}{2}}<\infty \}, \]

\noindent is called the \textbf{spectral dimension} of the spectral triple. 
\end{defn}

\subsection{Circles, curves, and sets built on curves}

The fundamental building block for spectral triples for fractal sets built on curves in \cite{CrIvLap} is the spectral triple for a circle.  Using circle triples, the authors of \cite{CrIvLap} construct spectral triples for an array of sets.   Let $C_{r}$ denote the circle with radius $r>0$.  In \cite{CrIvLap}, the \textbf{natural spectral triple} for the circle $ST_{n}(C_{r})= (AC_{r}, H_{r}, D_{r})$ is defined as follows:

\begin{enumerate}

\item[I.]  $AC_{r}$ is the algebra of complex continuous $2\pi r$-periodic functions on $\R$;

\item[II.] $H_{r}=L^{2}([-\pi r, \pi r], (1/2\pi r)\mu)$;

\item[III.] $D_{r}=-i\frac{d}{dx}$;

\item[IV.]  The representation $\pi$  sends elements of $AC_{r}$ to multiplication operators on $H_{r}$.

\end{enumerate}

  Note that $H_{r}$ has a canonical orthonormal basis given by $\exp\left(\frac{ikx}{r}\right)$, where $i=\sqrt{-1}$.  The operator $D_{r}$ is actually defined as the closure of the restriction of the above operator to the linear span of the basis. Then $D_{r}$ is self-adjoint and 

\[[D_{r},\pi_{r}(f)]=\pi_{r}(-iDf) \hspace{3mm} \mbox{  or just  } \hspace{3mm}  -iDf .\]

\noindent for any $C^{1}$ $2\pi  r$-periodic function $f$ on $\R$.  Thus the natural spectral triple is a spectral triple, and the eigenvalues of the Dirac operator are given as $\lambda_{k}=k/r$ for $k\in \mathbb Z$. \\

To use the circle triple as the basis for constructing spectral triples on more complex sets, it will be necessary to take countable sums of circle triples.  To avoid the problem of having $0$ as an eigenvalue with infinite multiplicity, the translated spectral triple is used in \cite{CrIvLap}:

\begin{enumerate}

\item[1.] Let $D^{t}_{r}=D_{r} + \frac{1}{2r}I$.

\item[2.] $ST_{t}(C_{r})=(AC_{r}, H_{r}, D_{r}^{t})$ is called the  \textbf{translated spectral triple} for the circle.

\end{enumerate}

The set of eigenvalues becomes $\{(2k+1)/2r: k\in \mathbb Z \}$, but the domain of definition stays the same and most importantly, as to not change the effect of the spectral triple,

\[[D^{t}_{r},\pi_{r}(f)]=[D_{r},\pi_{r}(f)] .\]

Let $d_{c}$ be the geodesic distance function on the circle.  Theorem 2.4 in \cite{CrIvLap} gives the following results:

\begin{itemize}

\item The metric induced by the spectral triple $ST_{n}(C_{r})$ coincides with the geodesic distance on $C_{r}$, i.e.,

\[d_{c}(s,t)=\sup\{|f(t)-f(s)|: ||[D_{r},\pi_{r}(f)]||\leq 1\}; \]

\item The circle triple is $p$-summable for any real $s>1$ but not summable for $s=1$, thus the spectral dimension of the spectral triple is $1$, coinciding with the dimension of a circle.

\end{itemize}

The interval is studied by means of the circle---by taking two copies of the  interval and gluing the endpoints together.  There is an injective homomorphism $\Psi$ from the continuous functions on an interval $[0,\alpha]$ to the continuous functions on $[-\alpha, \alpha]$ defined by 

\[\Psi_{\alpha}(f)(t)=f(|t|) .\]

The circle triple $(AC_{\alpha/ \pi}, H_{\alpha/ \pi}, D^{t}_{\alpha/ \pi})$ is then used to describe the spectral triple for $C([0,\alpha])$.  The fact that the following definition indeed defines a spectral triple follows immediately from the results on the circle:

 For $\alpha>0$, the \textbf{$\alpha$-interval spectral triple} $ST_{\alpha}=(\mathcal A_{\alpha}, H_{\alpha}, D_{\alpha})$ is given by the following data:

\begin{enumerate}

\item[i.] $\mathcal A_{\alpha}=C([0,\alpha])$;

\item[ii.] $\mathcal H_{\alpha}=L^{2}([-\alpha, \alpha],m/2\alpha)$, where $m/2\alpha$ is the normalized Lebesgue measure;

\item[iii.] the representation $\pi_{\alpha}:\mathcal A_{\alpha}\rightarrow B(\mathcal H_{\alpha})$ is defined for $f$ in $\mathcal A_{\alpha}$ as the multiplication operator which multiplies by the function $\Psi_{\alpha}(f)$;

\item[iv.] an orthonormal basis $\{e_{k}: k\in \mathbb Z\}$ for $\mathcal H_{\alpha}$ is given by $e_{k}=\exp(i\pi kx/\alpha)$ and $D^{t}_{\alpha}$ is the self-adjoint operator on $\mathcal H_{\alpha}$ which has all the vectors $e_{k}$ as eigenvectors and such that $D^{t}_{\alpha}e_{k}=(\pi (2k+1)/2\alpha)e_{k}$ for each $k\in \mathbb Z$.  Thus the eigenvalues of $D_{\alpha}^t$ are $\lambda_{k}=(\pi (2k+1)/2\alpha)$ for each $k\in \mathbb Z$.

\end{enumerate}

Let $d_{\alpha}(s,t)=|s-t|$ be the geodesic distance for the $\alpha$-interval.  Results for the $\alpha$-interval spectral triple, which follow immediately from the results for the circle, are stated in Theorem 3.3 in \cite{CrIvLap}:
 
\begin{itemize}

\item The metric induced by the $\alpha$-interval triple coincides with the geodesic distance for the $\alpha$-interval, i.e.,

\[d_{c}(s,t)=\sup\{|f(t)-f(s)|: ||[D_{\alpha},\pi_{\alpha}(f)]||\leq 1 \};\]

\item The $\alpha$-interval triple is $p$-summable for $s>1$ but not summable for $s=1$, thus it has spectral dimension $1$, coinciding with the dimension of the $\alpha$-interval.

\end{itemize}

Let $T$ be a compact Hausdorff space and $r:[0,\alpha]\rightarrow T$ a continuous injective mapping.  The image in $T$ will be called the continuous curve and $r$ the parameterization.  The \textbf{r-curve triple}, $ST_{r}$, is given by the interval triple as follows:\\

Let $r$ be as above and $(A_{\alpha}, H_{\alpha}, D^{t}_{\alpha})$ be the $\alpha$-interval spectral triple.  Then $ST_{r}=(C(T), H_{\alpha}, D^{t}_{\alpha})$ is an unbounded Fredholm module with representation $\pi_{r}: C(T)\rightarrow B(H_\alpha)$ defined via a homomorphism $\phi_{r}$ of $C(T)$ onto $A_{\alpha}$ given by

\begin{enumerate}

\item[a.] For all $f\in C(T)$, for all $t\in [0,\alpha]$, $\phi_{r}(f)(t) \vcentcolon= f(r(t))$;

\item[b.] For all $f\in C(T)$, $\pi_{r}(f) \vcentcolon= \pi_{\alpha}(\phi_{r}(f))$.

\end{enumerate}

\begin{rmk} \label{rmk1}
We will use the $r$-curve triple quite often;  so it is convenient to use the notation for its Dirac operator, $D_{r}=D^{t}_{\alpha}$.  Moreover, if there are curves $r_{j}$, and it is clear we are using the $r_{j}$-curve triples, then we will use $D_{j}=D_{r_{j}}$.  Note that from {\normalfont{(}}iv{\normalfont{)}} above, the eigenvectors of $D_{r}$ are  $e_{k}=\exp(i\pi kx/\alpha)$ with corresponding eigenvalues  $\lambda_{k}=(\pi (2k+1)/2\alpha)$, for each $k\in \mathbb Z$.
\end{rmk}

As is expected, the curve triple is summable for $s>1$ but not for $s=1$;  so its spectral dimension is $1$ (see Proposition 4.1 in \cite{CrIvLap}).  One can recover a metric distance on the image of the curve in $T$, of course dependent of parameterization.  If $T$ is a metric space, then a parameterization can be chosen so that the recovered metric distance coincides with the metric distance inherited from $T$ (see Proposition 4.3 in \cite{CrIvLap}).\\

The applications in \cite{CrIvLap} focused on sets built on curves, including finite collections of curves in a compact Hausdorff space, parameterized graphs, trees, and the Sierpinski gasket.   The general method for constructing triples for these sets given in \cite{CrIvLap} is by taking sums of triples for curves (circles, intervals).  Let $\left\{R_{j}\right\}_{j}$ be a collection of curves in a space $T$ (e.g., compact Hausdorff space, compact metric space, compact subspace of $\R^N$).  Then, following \cite{CrIvLap}, the triple for the union of these curves is, in general, given by\\

\[S=\left(C(T), \bigoplus_{j}H_{j}, \bigoplus_{j}D_{j} \right).\]

\noindent If $T$ is a compact Hausdorff space, then (even with a finite collection of rectifiable curves) it is not always true that $S$ is an unbounded Fredholm module.  However, in the case when there are finitely many rectifiable curves which pairwise intersect at finitely many points, $S$ is an unbounded Fredholm module (\cite{CrIvLap}, Prop. 5.1).  This type of construction is used in \cite{CrIvLap} for parameterized (finite) graphs, infinite trees, and for the self-similar Sierpinski gasket, with $T$ considered as a subspace of Euclidean space.  In the case of the Sierpinski gasket embedded in the $2$-dimensional Euclidean space $\R^{2}$, a countable sum of circle triples forms a spectral triple for the gasket.  The countable collection of circles corresponds to the countable collection of triangles, whose closure forms the Sierpinski gasket.  The spectral dimension is computed as $\log 3/\log 2$, which corresponds to its classic fractal dimension(s).  The spectral distance function recovers the Euclidean-induced geodesic distance and the (renormalized) standard measure on the gasket is recovered via the Dixmier trace \cite{CrIvLap}.  The construction of the spectral triple for the gasket and the recovering of its geometric data from the spectral triple is streamlined, due to its self-similarity.  One of the motivating factors for this article is to generalize such constructions and results to possibly non-self-similar sets built on curves, including the harmonic (Sierpinski) gasket which is perfectly suited to study analysis on the ordinary Euclidean (Sierpinski) gasket.  In addition, \propref{mainprop} in the current article unifies the constructions for many of the applications in \cite{CrIvLap}.  By considering $T$ as a compact length space in $\R^N$, and without any assumptions on the intersections of curves, we provide (in \propref{mainprop}) a spectral triple construction for a large class of sets built on countable collections of curves which includes both the Sierpinski gasket and the harmonic gasket (see Axiom 1 below).  In the next subsection, we conclude the preliminaries with definitions of the Sierpinski gasket and the harmonic gasket, as well as a discussion of the `measurable Riemannian geometry' of the Sierpinski/harmonic gasket.

\subsection{Sierpinski gasket and harmonic gasket}

The most common and intuitive presentation of the Sierpinski gasket is as a solid equilateral triangle which has a smaller equilateral triangle removed from its center, and again an even smaller triangle removed from each of the three remaining triangles and so on, ad infinitum, as seen in Figure 1.  This is done a countable number of times, and the closure of this process is called the Sierpinski gasket.  See the left side of Figure 2 for a high approximation of the Sierpinski gasket.\\

\begin{figure}[t]
\begin{center}
\begin{tabular}{lr} 
\includegraphics[scale=.51]{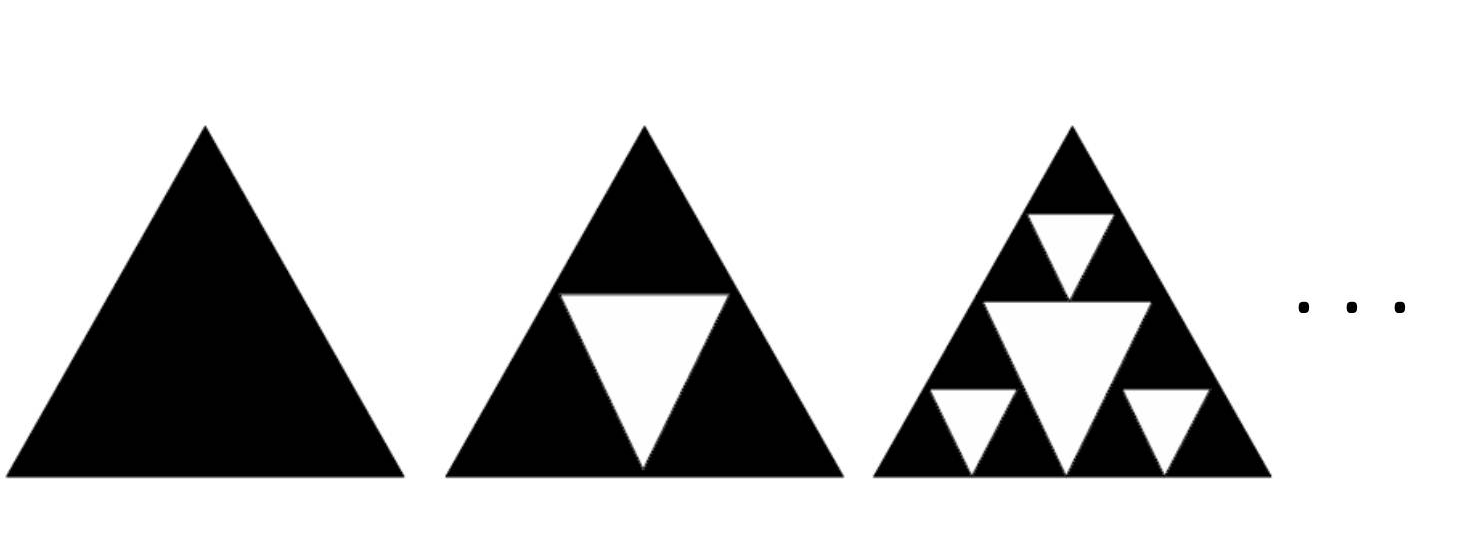} &  
\end{tabular}
\caption{\textbf{Figure 1.} Construction of the Sierpinski gasket by the removal of triangles}
\end{center}
\label{fig:RE}
\end{figure}

Considering the gasket in stages, or as approximations, is intuitive but also fundamental to defining additional structure on the gasket.  \emph{Graph approximations} will be the starting point for defining measure, operators, harmonic functions, etc. on the gasket.\\

The Sierpinski gasket is well described analytically as the unique fixed point of a certain contraction mapping on a metric space.  The contraction mapping to be defined is composed of three contraction mappings of $\R^{2}$ that will allow for analysis, not just on graph approximations, but on arbitrarily small \textsl{portions} of the gasket, called cells.  \\

Although continuity inherited from the Euclidean topology of the plane naturally connects with the analysis of the gasket, it is not critical to the definitions of measure, operators, harmonic functions, etc.  (In fact, it turns out that harmonic functions, defined exclusively in terms of graphs, are necessarily continuous functions in the Euclidean induced topology of the gasket.)  To generate the desired structure on the gasket, Euclidean neighborhoods are replaced with graph neighborhoods.  To begin, we define the following contractions on the plane:

\[F_{i}x=\frac{1}{2}(x-p_{i})+p_{i}\]

\[ i=1,2,3 \hspace{5mm} ; \hspace{5mm}  p_{i} \mbox{ is a vertex of a regular 3-simplex, } \mathcal{P}.\]

The Sierpinski gasket is the unique nonempty compact subset of $\R^{2}$ such that $K=\bigcup_{i=1}^{3}F_{i}(K)$.  For any integer $m\geq 1$, let $w$ be the multi-index given by $w=(w_{1},...,w_{m}), \hspace{5mm} w_{j}\in\{1,2,3\}$ and $F_{w}$ be given by $F_{w}=F_{w_{1}}\circ \cdots \circ F_{w_{m}}.$  Then $K$ satisfies $K=\bigcup_{|w|=m}F_{w}(K)$.  This is called the decomposition of $K$ into $m$-cells, with $F_{w}(K)$ being the $m$-cell given by $w$, denoted $K_{w}$.  Note that $K_w$ is a subset of $K$.  

The multi-index $w$ also provides a convenient addressing system for points of $K$, using \textsl{words} whose letters are elements of the set $S=\{1,2,3\}$.  Let $\Sigma= S^{\mathbb{N}}$, $W_{0}=\{\emptyset\}$, and $W_{m}= S^{m}$ for $m>0$.  (Note that for $m\geq 0$,  $W_m$ is the set of all words of length $m$.)  The set of all words of finite length is $W^{*}=\bigcup_{m\geq 0}W_{m}$.  To describe the identification of words with points of $K$, it is useful to define the \textsl{vertices} $V^*$ of $K$ given by $V^{*}= \bigcup_{m\geq 0}V_{m}$, where $V_{0}=\mathcal{P}=\{p_{1}, p_{2}, p_{3}\}$ and $V_{m}= \bigcup_{w\in W_{m}}F_{w}(V_{0})$.  Consider $\Sigma$, the set of infinite words, to have the standard metric topology on sequences and $K$ to have the Euclidean topology inherited from the plane.  Then it is well known (see, e.g., \cite{Kig3, Strichartz}) that there is a continuous surjection $\pi: \Sigma \rightarrow K$ such that

\[\pi(w)=\bigcap_{m\geq 0}K_{w_{1},...,w_{m}} \]

\noindent and 

\[|\pi^{-1}(x)|= \left\{\begin{array}{ll}

2, & x\in V^{*}-V_{0}\\ 

1, & otherwise. 

\end{array} \right. \]

Graph approximations of $K$ and their associated vertices are central to all further analysis of the gasket.  The \textsl{$m^{th}$-level graph approximation}, $\Gamma_{m}$, is given by $\Gamma_{m}=\bigcup_{|w|=m}F_{w}(V_{0})$ and has vertices $V_m$.  Thus $V^*$ is the union of the vertices of all graph approximations.  A \textsl{graph cell}, $\Gamma_w$, is defined as $\Gamma_w=V_{0}$ for $|w|=0$ and $\Gamma_w=F_w(V_{0})$ for $|w|>0$.  The transition from analysis on graphs to analysis on $K$ comes readily since $V^*$ is a dense subset of $K$.  The functions on $K$ that we will consider will be continuous (in the Euclidean subspace topology) and therefore they will be completely determined by their values on the collection of vertices.  Graph cells, however, play a special role in certain spectral triple constructions, particularly in their identification with the triangles they define in $\R^2$.   For this reason, we will always identify a graph cell $\Gamma_w$, as a set, with the triangle it defines in $\R^2$. \\ 

The energy form on $K$ is constructed from \textsl{graph energies}, independent of a notion of Laplacian or differential operators.  The \textsl{graph energy form on $\Gamma_{m}$}, $\Ec_m[u,v]$, is given by (\cite{Kig3, Strichartz})

\[\Ec_{m}(u,v)=\left(\frac{5}{3}\right)^{m}\sum_{p\cong q: p,q\in V_{m}}{(u(p)-u(q))(v(p)-v(q))},\]

\noindent where $V_m$ is the set of vertices of $\Gamma_m$ and for $p,q \in V_m$, the notation $p\cong q$ means that $p$ and $q$ are neighbors in the finite graph $\Gamma_m$.  \textsl{The energy form}, $\Ec(u,v)$, on $K$ is then given by $\Ec(u,v)=\lim_{m \rightarrow \infty}\Ec_{m}(u,v)$ with the \textsl{energy}, $\Ec(u)$, on $K$ given by
$\Ec(u)=\Ec(u,u)$. \\

Since $\Ec_{m}$ is a non-decreasing sequence, the above limit defining $\Ec(u,v)$ exists and is finite by design, for all $u,v\in \dom\Ec=\{u\in C(K) \hspace{3mm} | \hspace{3mm} \lim_{m\rightarrow \infty}\Ec_{m}(u,u)<\infty\}$.  The expression for the graph energies has several motivations.  Kusuoka in \cite{Kus1, Kus2}, and Goldstein in \cite{Gold}, have independently constructed the Brownian motion on the Sierpinski gasket as a scaling limit of random walks.  To view the energy as an analytic counterpart to Brownian motion on the gasket, one must think of it is a Dirichlet form (see \cite{Barlow, Kig1, Kig2, Kig3, Kus1, Kus2}).  Other physical interpretations of the energy are provided in terms of electrical resistance networks (see \cite{Kig2, Kig3}), as well as in terms of systems of springs attached to point masses assigned to graph vertices 
(see \cite{Ram1, Ram2} and, e.g., \cite{Strichartz}).  \\

The theory of harmonic functions on $K$ is a generalization of classical harmonic theory in which there are the standard equivalences:  1) $u$ is harmonic; 2) $u$ is an energy minimizer, for given boundary values; 3) $u$ has the mean value property;  4) $\Delta u = 0$.  A suitable springboard for harmonic theory on $K$ is that of energy minimization.  It is the case that a harmonic function defined in this way will enjoy a mean value property as well the Laplacian condition.    To be precise, let $u$ be defined on $V_{0}$.  (Here, $V_{0}$ should be thought of as the `boundary' of $K$.)  Then, there is a unique extension of $u$ from $V_{0}$ to $V_{m+1}$, denoted $\hat{u}$, which minimizes the energy $E_{m+1}$ with the relation

\[E_{0}(u)=\left(\frac{5}{3}\right)^{m}E_{m+1}(\hat{u}). \]

\noindent The function $\hat{u}$ is called the \textsl{harmonic extension} of $u$.  Given values of a function $u$ on $V_{0}$, $u$ can be uniquely extended harmonically to $V_{m}$ for any $m$ and therefore can be extended to $V^{*}$.  The function $\hat{u}$, defined in this way, is (uniformly) continuous on $V^{*}$ which is dense in $K$ with respect to the Euclidean inherited topology.  Hence, $\hat{u}$ extends uniquely to a function $u$ on $K$, called a \textsl{harmonic function} on $K$.\\

Note that the harmonic function $u$, is uniquely determined by its boundary value, $u|_{V_{0}}$.  Let the space of harmonic functions be denoted by $\mathcal{H}$. In this case, $\mathcal{H}$ forms a $3$-dimensional linear space which we can identify with $\R^{3}$ by associating $u\in \mathcal{H}$ to the triple $(u(p_{1}),u(p_{2}),u(p_{3}))$ in $\R^{3}$.  Moreover, modding out $\mathcal{H}$ by the constant functions on $K$, we have
$\mathcal{H}/\{\mbox{constant functions}\} \cong \R^{3}/\{\Span(1,1,1)\}$.  Note that the right side is the $2$-dimensional subspace of $\R^{3}$,

\[M_{0} \vcentcolon= \{(x,y,z)\in \R^3 \hspace{2mm}|\hspace{2mm} x+y+z=0 \}.\]

 The Sierpinski gasket is not a smooth, nor even a topological manifold; yet, we can look at it as a space to be geometrized.  The analysis has been based on graphs and the neighbor relation so that the bending, stretching, and twisting of $K$ away from how it sits in the flat plane, while preserving the neighbor relations of vertices, does not affect the analysis.  So even though the standard visualization of $K$ is in the plane, this perspective begs to see $K$ as a more abstract object, \textsl{awaiting} a metric.\\
 
    In this section, $K$ is assigned or geometrized by the \textsl{harmonic metric} to become the `geometric' space called the \textsl{harmonic gasket} (or sometimes, the \textsl{harmonic Sierpinski gasket}) and denoted $K_{H}$, a particular geometric realization of $K$. The latter perspective hints at $K$ and $K_{H}$ as being distinct spaces equipped with their own geometries:  $K$ with the geometry implied by its specific manner of inclusion in the Euclidean plane, and $K_{H}$ with the geometry implied by its configuration in the plane $M_{0}$ in $\R^{3}$.  The harmonic gasket is defined using the space of harmonic functions, $\mathcal{H}$.  Recall that a harmonic function, $h$, is determined uniquely by its values on $V_{0}$.  Identifying $\mathcal{H}$ with $\R^{3}$, take $\{h_{1}=(1,0,0), \hspace{5mm} h_{2}=(0,1,0), \hspace{5mm} h_{3}=(0,0,1)\}$,  as a basis for $\mathcal{H}$.  In terms of the evaluation of harmonic functions, this is equivalent with $h_{i}(p_{j})=\delta_{ij}$ for $i,j=1,2,3$ and $p_{j}\in V_{0}$. The final step in the construction of the harmonic gasket is to use $h_{1},h_{2},$ and $h_{3}$ as a single `coordinate chart' for $K$ in the plane $M_{0}$.  Kigami \cite{Kig1} (see also \cite{Kig4}) defines the following map,

\[\Phi: K\rightarrow M_{0}\]

\[\mbox{by  }\hspace{5mm} \Phi(x)=\frac{1}{\sqrt{2}}\left(\left(\begin{array}{l} h_{1}(x)\\h_{2}(x)\\h_{3}(x)\end{array}\right) -\frac{1}{3}\left(\begin{array}{l} 1\\1\\1\end{array}\right)\right), \]

\noindent which is a homeomorphism onto its image (see Figure 2).  Then $K\cong \Phi(K)\equiv K_{H}$ defines the \textsl{harmonic gasket} or \textsl{Sierpinski gasket in harmonic metric}.  Though $K_{H}$ is not a self-similar fractal, it is \textsl{self-affine} and can be given as the unique fixed point of a certain contraction mapping, induced by the iterated function system $\{H_{i}\}_{i=1}^{3}$ defined below.  The homeomorphism $\Phi$ preserves compactness, so that $K_{H}$ is a compact subset of $M_{0}$.  To be precise, let $P$ be the orthogonal projection from $\R^{3}$ to $M_{0}$.  Let 

\[q_{i}=\frac{P(e_{i})}{\sqrt{2}} \mbox{ for  } i=1,2,3, \]

\noindent where $\{e_{i}\}_{i=1}^{3}$ is the standard basis for $R^{3}$.  The  $q_{i}$'s form a $3$-simplex in $M_{0}$.  For each $i=1,2,3$, choose $f_{i}$ such that

\[\left\{\frac{q_{i}}{|q_{i}|},f_{i}\right\}\]

\noindent gives an orthonormal basis for $M_{0}$.  Also, define the maps $J_{i}: M_{0}\rightarrow M_{0}$ by

\[J_{i}(q_{i})=\frac{3}{5}q_{i}  \mbox{  and  } J_{i}(f_{i})=\frac{1}{5}f_{i}. \]

\noindent Using the $J_{i}$'s, define the following contractions $H_{i}:M_{0}\rightarrow M_{0}$ by $H_{i}(x)=J_{i}(x-q_{i}) + q_{i} \mbox{  for  } i=1,2,3.$  The harmonic gasket, $K_{H}$, is then the unique nonempty compact subset of $M_{0}$ such that $K_{H}=\bigcup_{i=1}^{3}H_{i}(K_{H})$.  \textsl{Recall that, unlike $K$, which is self-similar, $K_{H}$ is only self-affine.}  The contractions $H_{i}$ naturally relate to the contractions $F_{i}$ used to define $K$ via the homeomorphism $\Phi$ which commutes with the contractions, in the sense that $\Phi\circ F_{i} = H_{i}\circ \Phi$ for $i=1,2,3$.  The graph approximations of $K_{H}$ can be attained through $\Phi$ from the $F_{i}$'s or directly from the $H_{i}$'s as in the case of $K$.  See Figure 2 for a comparison of the Sierpinski and harmonic gaskets.\\

In the sequel, we denote by $J_{w}$ the linear map obtained by composing the $J_{i}$'s corresponding to the finite word $w$.  Specifically, $J_{w}=J_{w_{1}} \circ \cdots \circ J_{w_{m}}$ if $w=w_{1} \cdots w_{m} \in W_{m}$.

\begin{figure}
\framebox{
\begin{tabular}{ccc}

\hspace{5mm} Sierpinski Gasket & \hspace{4mm} homeomorphism & \hspace{-4mm} harmonic gasket\\

\begin{minipage}[c]{3.94cm}
\begin{center} \hspace{-2mm}
\includegraphics[scale=.23]{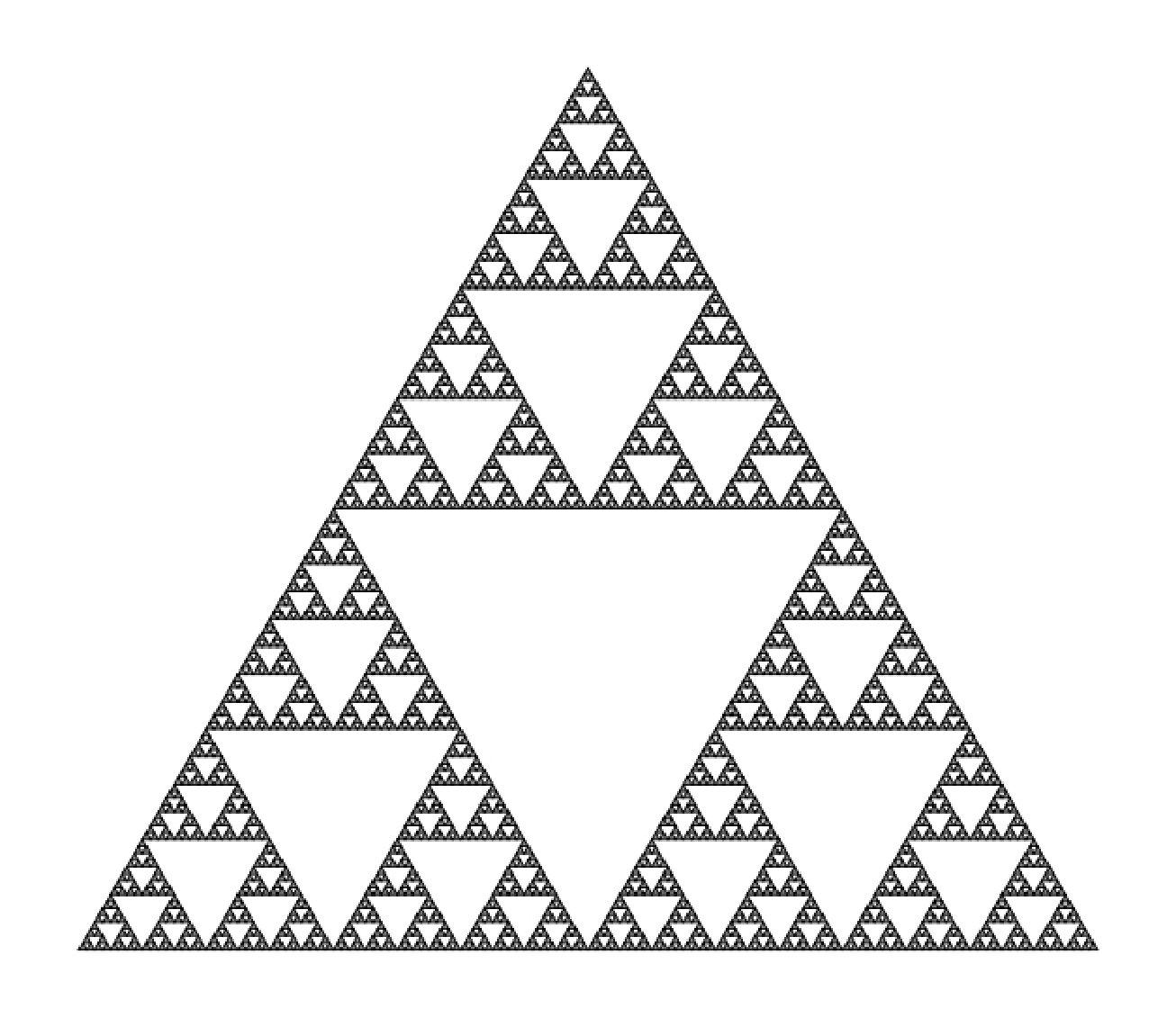}
\end{center}
\end{minipage}

 & 

\begin{minipage}[c]{.94cm} 

\begin{center} \Huge{$\stackrel{\Phi}{\longrightarrow}$} \end{center} 

\end{minipage}  

&

\begin{minipage}[c]{3.94cm}
\begin{center} \hspace{-6mm}
\includegraphics[scale=.15]{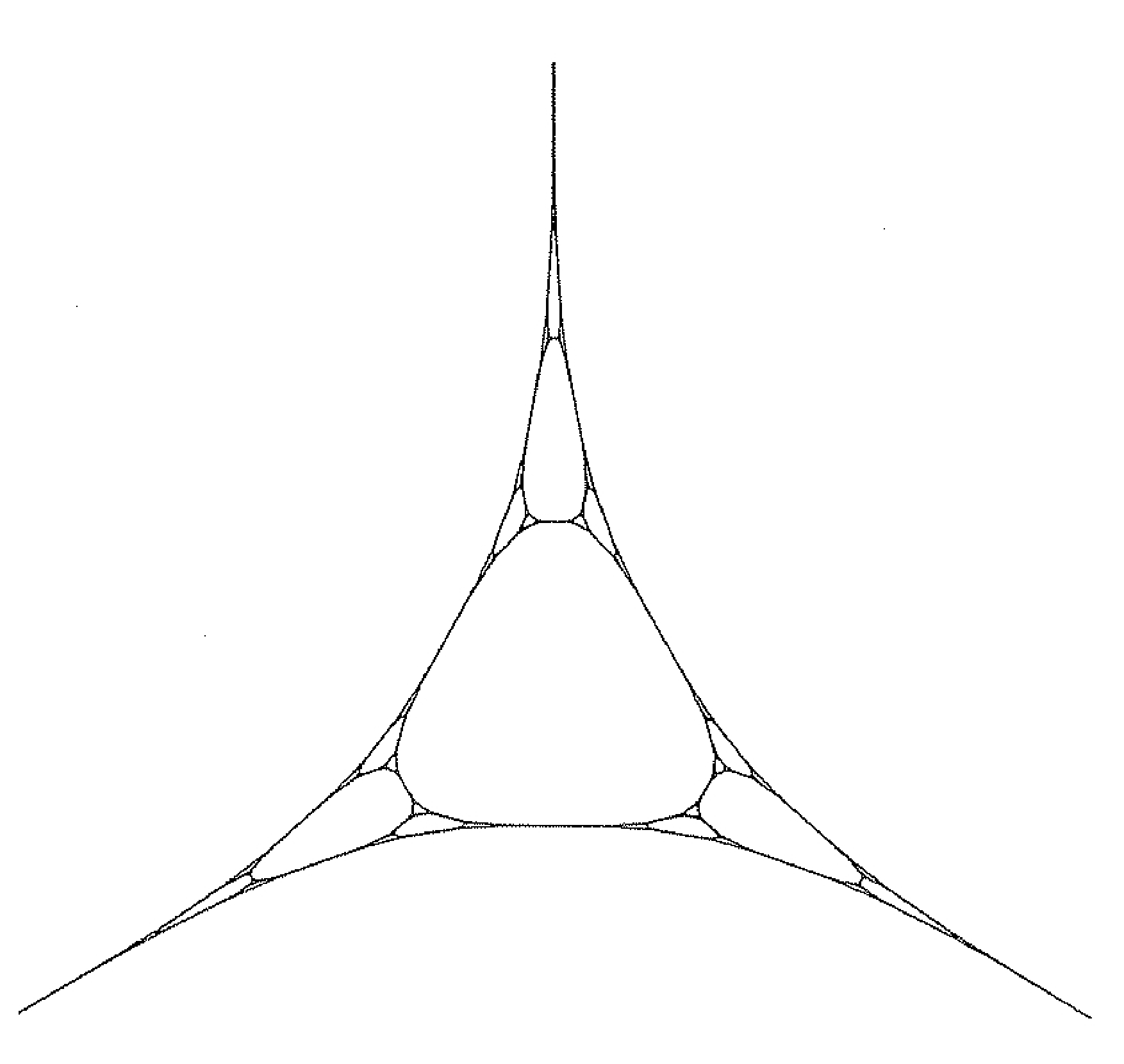}
\end{center}
\end{minipage}

\end{tabular}}

\caption{\textbf{Figure 2.} The Sierpinski gasket $K$ is pictured on the left and the `harmonic gasket' $\KH$ pictured on the right is the homeomorphic image of $K$ by $\Phi$, which is a `harmonic coordinate chart' for the Sierpinski gasket. }

\end{figure}

\subsection{Measurable Riemannian geometry}

The primary ingredients of Kigami's prototype for a \textsl{measurable Riemannian geometry} are the \textsl{measurable Riemannian structure} and geodesic distance;  see \cite{Kig4}.  The measurable Riemannian structure is due to Kusuoka \cite{Kus1} and  is a triple $(\nu,Z,\nablaT)$, where $\nu$ is the Kusuoka measure on $K$, $Z$ is a non-negative symmetric matrix, and $\nablaT$ is an operator analogous to the Riemannian gradient.  More precisely, Kusuoka has shown in \cite{Kus1, Kus2} that for any $u$ and $v$ in the domain $\dom\Ec$ of the energy functional on the Sierpinski gasket, $K$, we have\\

\[\Ec(u,v)=\int_{K}{(\nablaT u, Z\nablaT v)d\nu}, \]

\noindent where $Z$, $\nablaT u$, and $\nablaT v$, are $\nu$-measurable functions defined $\nu$-a.e. on $K$;  see also \cite{Kig1} and \cite{Kig4}.  The equality above is analogous to its \textsl{smooth} counterpart in Riemannian geometry, and thus gives validity to the title `measurable Riemannian structure' for $(\nu,Z,\nablaT)$.  Here, the Kusuoka measure $\nu$ is the analog of the Riemannian volume and $Z$ is the analog of the Riemannian metric.  In \cite{Kig4}, Kigami furthers the likeness to Riemannian geometry by introducing a notion of smooth functions on $K$, as well as a theorem relating the Kusuoka gradient to the usual gradient on the Euclidean plane (see also \cite{Kig1}), and a notion of geodesic distance on $K$, which is realized by a $C^{1}$ path in the plane.  \\

The Sierpinski gasket, in Euclidean or standard metric does not have $C^{1}$ paths between points, in general.  In order to get $C^{1}$ paths, Kigami views the gasket in harmonic coordinates, as the harmonic gasket described earlier.  The harmonic gasket, $K_{H}$, does have $C^{1}$ paths between any two points.  Then via the homeomorphism, $\Psi$, a geodesic distance, realized by a $C^{1}$ path on $K_{H}$ of minimal length, is attached to $K$.  (It is noteworthy that such a geodesic path is $C^{1}$ but not usually $C^{2}$;  see \cite{Tep1}.)\\

 The Kusuoka measure is the measurable analog of Riemannian volume.  The existence of the Kusuoka measure $\nu$ on $K$ is due to Kusouka \cite{Kus1}.  Further details on the Kusuoka measure can be found in \cite{Kus1}, \cite{Kus2}, \cite{Kig1}, \cite{Kig4}, \cite{Tep1} and \cite{Tep2}.\\
 
The measurable analog of the Riemannian metric, or the \textsl{measurable Riemannian metric} $Z$, is also due to Kusuoka \cite{Kus1, Kus2}.  In Proposition 2.11 of \cite{Kig4}, the definition of $Z$ is given as follows:  Let $w\in W_{m}$ and define $Z_{m}(w)=J_{w}^{t}(J_{w})/||J_{w}||_{HS}^{2}$, where $J_{w}^{t}$ is defined in terms of the transpose (or the adjoint) of $J_{w}$ and $||J_{w}||_{HS}$ denotes the Hilbert--Schmidt norm of $J_{w}$.  Then $Z(w)=\lim_{m\to\infty}Z_{m}(w_{1}...w_{m})$ exists $\nu$-a.e. for $w\in \Sigma$, rank$Z(w)=1$ and $Z(w)$ is the orthogonal projection onto its image for $\nu$-a.e. $w\in \Sigma$. \\

 In order to define the metric on $K$, let $Z_{*}(x)=Z(\pi^{-1}(x))$, where $\pi$ was defined in Subsection 2.3.  Then $Z_{*}$ is well defined, has rank 1 and is the orthogonal projection onto its image for $\nu$-a.e. $x\in K$.  Similar as with the Kusuoka measure, the $*$ is dropped and $Z$ is used instead of $Z_{*}$.  It also holds that $Z$ is well defined on $V_{*}$, since for $x\in V_{*}$ and $\pi^{-1}(x)=\{w,\tau\}$, we have $Z(w)=Z(\tau)$;  see \cite{Kig4}.\\

There are a few characterizations of the gradient in the setting of the measurable Riemannian structure. The first we will mention is due to Kusuoka \cite{Kus1}.  In Theorem 2.12 in \cite{Kig4}, Kigami gives Kusuoka's result which is the existence of an assignment $\nablaT: \dom \Ec\rightarrow \{Y \hspace{1mm} | \hspace{1mm} Y: K\rightarrow M_{0}, \hspace{1mm} Y \mbox{  is  } \nu\mbox{-measurable}\}$ such that 

\[\Ec(u,v)=\int_{K}{(\nablaT u, Z\nablaT v)d\nu}, \]

\noindent for any $u,v\in \dom \Ec$.   \\

Kigami's approach to the gradient on $K$ is to start with the usual gradient on open subsets of the plane $M_{0}$ which contain $K_{H}$.  More precisely, fixing an orthonormal basis for $M_{0}$ and identifying $M_{0}$ with $\R^{2}$,  the gradient on $M_{0}$ is given by $\nabla u= {}^t(\partial u/\partial x_{1},\partial u/\partial x_{2})$.  In Proposition 4.6 of \cite{Kig4}, it is shown that if $U$ is an open subset of $M_{0}$ which contains $K_{H}$, $v_{1},v_{2}\in C^{1}(U)$, and $v_{1}|_{K_{H}}=v_{2}|_{K_{H}}$, then $(\nabla v_{1})|_{K_{H}}=(\nabla v_{2})|_{K_{H}}$.  In this sense, the gradient of a \textsl{smooth function} on $K_{H}$ is well defined by the restriction of the usual gradient to an open subset of $M_{0}$.  Then, using $\Phi$, this theory can be pulled back to $K$. Precisely, in \cite{Kig4}, Kigami defines the space $C^{1}(K)$ given by
\[\{u : u=(v|_{K_{H}})\circ \Phi, \mbox{  where  } v \mbox{  is  } C^{1} \mbox{  on an open subset of }  M_{0} \mbox{   containing  }  K_{H} \}\] 

\noindent and for $u\in C^{1}(K)$, 

\[\nabla u = (\nabla v|_{K_{H}})\circ \Phi.  \]

\noindent In Theorem 4.8 of \cite{Kig4}, the following results are established:

\begin{enumerate}

\item $C^{1}(K)$ is a dense subset of $\dom\Ec$ under the norm $||u||=\sqrt{\Ec(u,u)} +||u||_{\infty, K}$;

\item $\nablaT u=Z\nabla u$ for any $u\in C^{1}(K)$;

\item $\displaystyle{\Ec(u,v)=\int_{K}{(\nabla u, Z\nabla v)d\nu}}$ \hspace{5mm} for any $\displaystyle{u,v\in C^{1}(K)}$.

\end{enumerate}

\noindent Thus Kigami shows that his gradient, $\nabla$, `essentially' coincides with the Kusuoka gradient, $\nablaT$---at least up to its role in the energy formula.  For related representations of the gradient for the harmonic gasket, see \cite{Tep1}. \\ 

The first important theorem regarding a geodesic, or segment, or shortest path between two points on $K$, in the context of $K$ in harmonic coordinates, is due to Teplyaev.  First, a boundary curve $\tau$ of the gasket in harmonic coordinates, is defined by Teplyaev as a parameterization of a boundary of a connected component of $M_{0}\backslash K_{H}$.   In Theorem 4.7 of \cite{Tep1}, Teplyaev states the following: 

\begin{enumerate}

\item $\tau$ is concave and is a $C^{1}$ curve but \textsl{is not a} $C^{2}$ curve; 

\item for any $x\in K$ such that $\Psi(x)\in \tau$, the projection $P_{x}$ is, in harmonic coordinates, the orthogonal projection onto the tangent line to $\tau$. 

\end{enumerate}

 Let 
  \[h_{*}(p,q) \vcentcolon= \inf\{\hspace{1mm} l(\gamma)\hspace{1mm} | \hspace{1mm} \gamma \mbox{  is a rectifiable curve in  } K_{H} \mbox{  between  } p \mbox{  and  } q\},\]
 
 \noindent where $l(\gamma)$ is the length of the curve $\gamma$.  Kigami makes use of the above results to prove Theorem 5.1 in \cite{Kig4} which states that for any $p,q\in K_{H}$, there exists a $C^{1}$ curve $\gamma_{*}: [0,1]\rightarrow K_{H}$ such that $\gamma_{*}(0)=p$, $\gamma_{*}(1)=q$, $Z(\Phi^{-1}(\gamma_{*}(t)))$ exists and $\frac{d\gamma_{*}}{dt}\in ImZ(\Phi^{-1}(\gamma_{*}(t)))$ for any $t\in [0,1]$, and 
 
\[h_{*}(\gamma_{*}(a),\gamma_{*}(b))=\int_{a}^{b}{\left(\frac{d\gamma_{*}}{dt},Z(\Phi^{-1}(\gamma_{*}(t)))\frac{d\gamma_{*}}{dt}\right)^{\frac{1}{2}}dt} \]

\noindent for any $a,b\in [0,1]$ with $a<b$.  Note that due to this result, the infimum in the definition of $h_{*}$ can be replaced by the minimum. Kigami calls $\gamma_{*}$ a \textsl{geodesic} between $p$ and $q$.  The proof of this theorem is lengthy, with the majority of the work going into proving Theorem 5.4 of \cite{Kig4}.  Kigami credits Teplyaev (Theorem 4.7 in \cite{Tep1}) for the latter result but gives his own proof.  He uses the distance function $h_{*}$ in order to define the \textsl{harmonic shortest path metric} on $K$,  $d_{*}(.,.)$, for $x,y\in K$, as
 
 \[d_{*}(x,y)=h_{*}(\Phi(x),\Phi(y)), \]
 
\noindent or with slight abuse,

\[d_{*}(x,y)=\int_{a}^{b}{\left\langle \dot{\gamma},Z\dot{\gamma}\right\rangle}^{\frac{1}{2}}dt, \]
 
\noindent where $\gamma$ is a geodesic (shortest path) between $x$ and $y$ in $K_{H}$ with $\gamma(a)=x$ and $\gamma(b)=y$.  This latter representation provides a strong analogy with the geodesic distance on a Riemannian manifold, where a smooth metric has been replaced by a measurable metric, $Z$.\\

Note that $d_*$ corresponds to the geodesic metric on the harmonic gasket $\KH$.  Further note that, clearly, a geodesic between any two points of $\KH$ (or, equivalently, a shortest harmonic path between any two points of $K$) is usually not unique.  (See Figure 2.)

\section{Spectral Geometry of Fractal Sets}

This section is motivated by the desire to specify a natural or intrinsic metric on certain sets built on curves, including certain fractal sets (and certain infinite graphs), via a Dirac operator and its associated spectral triple.  In this section, we look at a class of sets built on curves in $\R^{n}$, each of which is assumed to have a shortest path metric which we will call the geodesic distance.  The construction of the Dirac operator detailed in this section is a generalization of a construction for a finite collection of curves used in \cite{CrIvLap}.  We show that the spectral distance function induced by the Dirac operator recovers the geodesic distance for this class of sets.  The Sierpinski gasket and the harmonic gasket both fall in this class of sets, whereas only the former example lies within the scope of \cite{CrIvLap}.  The harmonic gasket as well as alternate constructions for the Dirac operator for the harmonic gasket are discussed in the next section.  First we recall some definitions related to length spaces and a relevant result, the Hopf--Rinow Theorem (see, e.g., \cite{Gromov} for the general case of length spaces and \cite{Peterson} for the original case of Riemannian manifolds):\\

\begin{defn}  Let $(M,d)$ be a metric space.  The \textbf{induced intrinsic metric} $d_{I}=d_{I}(x,y)$ is defined as the infimum of the $d$-induced lengths of {\normalfont{(}}continuous{\normalfont{)}} paths from $x$ to $y$.  When there is no path from $x$ to $y$, then $d_{I}(x,y)$ is defined to be infinite. If $d(x,y)=d_{I}(x,y)$ for all $x$ and $y$ in $M$, then $(M,d)$ is called a \textbf{length space} and the metric $d$ is said to be \textbf{intrinsic}. 
\end{defn}

\begin{defn} Let $(M,d)$ be a length space and $\gamma: I\rightarrow M$ be a continuous path parameterized by arclength, where $I$ is an interval of the reals.  If $d(\gamma(t_{1}), \gamma(t_{2}))=|t_{1}-t_{2}|$ for all $t_{1}$ and $t_{2}$ in $I$, then $\gamma$ is called a \textbf{minimizing geodesic} or \textbf{shortest path}.

\end{defn}

\begin{thm}{\normalfont{(}}\textbf{Hopf--Rinow}{\normalfont{)}}  \label{HR}
If a length space $(M, d)$ is complete and locally compact, then any two points in $M$ can be connected by a minimizing geodesic, and any bounded closed set in $M$ is compact.
\end{thm}

 Let $X\subset \R^{n}$ be a compact length space.  Then $X$ is necessarily complete.  Furthermore, by the Hopf--Rinow Theorem (\thmref{HR}), $X$ has minimizing geodesics.  Let $L(\gamma)$ denote the length of the continuous curve $\gamma$ parameterized by its arclength.  We consider the following axioms for $X$:

\begin{enumerate}

\item[]  \textbf{Axiom 1 }:   $X=\overline{\mathcal{R}}$, where $\mathcal{R} =\bigcup_{j\in \N}R_{j}$;  $R_{j}$ is a rectifiable $C^{1}$ curve for each $j\in \N$, with $L(R_{j})\rightarrow 0$ as $j\rightarrow \infty$.

\item[]  \textbf{Axiom 2 }:  There exists a dense set $\mathcal{B}\subset X$ such that for each $p\in \mathcal{B}$ and each $q\in X$, one of the minimizing geodesics from $p$ to $q$ can be given as a countable (or finite) concatenation of the $R_{j}$'s.

\end{enumerate}
 
\begin{rmk} \label{rmk2}
In Axiom 2, it is understood that the countable concatenation of $R_{j}$'s begins with $p\in \mathcal{B}$ as the initial endpoint of some $R_{j}$.  Therefore, $\mathcal{B}$ is a subset of the collection of endpoints of the $R_{j}$ curves, and hence, Axiom 2 implies that the endpoints are dense in $X$. 
\end{rmk}

  For $p,q\in X$ and $\gamma$ a minimizing geodesic between $p$ and $q$, we will define the \textsl{geodesic distance}, $d_{geo}$ , by $d_{geo}(p,q)=L(\gamma)$.  
 

\begin{prop} \label{mainprop}
Suppose $X$ is a compact length space which satisfies Axiom 1.  Then the countable sum of $R_{j}$-curve triples, $S(X)$, is a spectral triple for $X$.  Furthermore, if $D$ is the Dirac operator associated to $S(X)$ and $L(R_{j})=\alpha_j$ for each $j\in \N$, then the spectrum of $D$ is given by

\[\sigma(D)=\bigcup_{j\in \N}\left\{\left[\frac{(2k+1)\pi}{2\alpha_{j}}\right] : k\in \mathbb Z \right\}.\] 

Moreover, the spectral dimension of $X$ with respect to $S(X)$ {\normalfont{(}}or equivalently, the metric dimension of $S(X)${\normalfont{)}} is given by

\[\mathfrak{d}_{S(X)}=\inf \left\{p>1: \sum_{j\in \N}\alpha_{j}^{p} < \infty \right\}    .\]
\end{prop}

\begin{proof}
For each $j\in \N$, let $R_{j}$ be parameterized such that $L(R_j)=\alpha_j$.  Using the $r$-curve triple with $r=R_{j}$ and $\alpha=\alpha_{j}$ yields the unbounded Fredholm module 
 
 \[S_{j}=\left(C(X), H_{j}, D_{j}\right)\]

 \noindent for $R_{j}$, with representation $\pi_{j}$.  To construct a spectral triple for $X$, we define 

\[\bigoplus_{j\in \N}S_{j}=\left(C(X), \bigoplus_{j\in \mathbb N}H_{j}, \bigoplus_{\j\in \mathbb N}D_{j}\right) ,\]

\noindent with representation

\[\bigoplus_{j\in\N}\pi_{j}.\]

\noindent We refer to $S(X)$ as the countable sum of $R_{j}$-triples, with the notation

\[S(X)=\bigoplus_{\j\in \N}S_{j}, \hspace{5mm} D=\bigoplus_{\j\in \N}D_{j},  \hspace{5mm} H=\bigoplus_{j\in \N}H_{j}, \hspace{5mm} \mbox{and} \hspace{5mm} \pi_X=\bigoplus_{j\in\N}\pi_{j},\]

\noindent so that $S(X)=(C(X), H, D)$.\\

   By the Stone--Weierstrass Theorem, the real linear functionals on $\R^{n}$ are a dense subset of $C(X)$.  The real functionals will suffice as a dense subset having bounded commutators with the Dirac operator $D$.  Indeed,  if $f(x_{1},...,x_{n})=a_{1}x_{1}+...+a_{n}x_{n}$ is an arbitrary real functional and $R_{j}$ is parameterized (by arclength) in the variable $\tau$,  then (letting $||.||_{\infty}\vcentcolon = ||.||_{\infty, \R^{n}}$, $i\vcentcolon =\sqrt{-1}$ and using the discussion following Formula 1 in Subsection 2.1 in order to justify the first two equalities), we obtain

\[||[D_{j},\pi_{j}(f)]||=||D_{j}(f)||=||D_{j}(f)||_{\infty}=\left|\left|{\frac{1}{i}\frac{df}{d\tau}}\right|\right|_{\infty}\]

\[=||a_{1}(x'_{1}(\tau))+ ... +a_{n}(x'_{n}(\tau))||_{\infty}\leq |a_{1}|+...+|a_{n}|.\]

\noindent Since this bound is not dependent on $j$, we have 

\[||[D,\pi_X(f)]||=\sup_{j}\left\{||[D_{j},\pi_{j}(f)]||\right\}\leq  |a_{1}|+...+|a_{n}|.\] 

\noindent Therefore, as claimed above, the real linear functionals on $\R^{n}$ form a dense subspace of $C(X)$ comprised of elements having bounded commutators with $D$.\\

The eigenvalues of $D_{j}$ are determined by the length $\alpha_{j}$ of $R_j$ and are given in \rmkref{rmk1} of Subsection 2.2 above as $(\pi (2k+1)/2\alpha_j)$ for $k\in \mathbb Z$.  The eigenvalues of $D$ are the disjoint union of the eigenvalues for the $D_{j}$'s; so

\[\sigma(D)=\bigcup_{j\in \N}\left\{\left[\frac{(2k+1)\pi}{2\alpha_{j}}\right] : k\in \mathbb Z \right\}.\] 

Since $\alpha_{j}\rightarrow 0$ as $j\rightarrow \infty$, we deduce that $(D^{2}+I)^{-1}$ is a compact operator.  The self-adjointness of $D$ follows from the fact that its summands $D_{j}$ are self-adjoint for each $j$.   Thus $S(X)$ is an unbounded Fredholm module.  Furthermore, since a function in the image of $\pi_X$ is densely defined on $X$, the representation is faithful, so that $S(X)$ is a spectral triple.  (See Definitions 2.1 and 2.2 in Subsection 2.1.)  \\

Using the expression for the spectrum $\sigma(D)$ obtained above, we see that the spectral dimension (Definition 2.3 in Subsection 2.1) is given by

\[\mathfrak{d}_{S(X)}=\inf \left\{p>0: \sum_{j \in \N}\sum_{k\in \mathbb{Z}}{\left|\frac{(2k+1)\pi}{2\alpha_{j}}\right|^{-p}}<\infty\right\}.\]

Now, the double sum over $j$ and $k$ is finite if and only if the sum over $k$, $\sum_{k\in \N}(2k+1)^{-p}$, and the sum over $j$, $\sum_{j\in \N}\alpha_{j}^{p}$, are both finite.  (Indeed, up to a trivial multiplicative factor, the double sum can be written as the product of these two single sums.)  Since, clearly, $\sum_{k\in \N} (2k+1)^{-p}<\infty$ if and only if $p>1$, it follows that 

\[\mathfrak{d}_{S(X)}= \inf \left\{p>1 : \sum_{j\in \N} \alpha_{j}^{p}<\infty \right\},\]

\noindent as desired.
\end{proof}

\begin{rmk}  \label{rmk3}
It follows from the expression obtained for $\mathfrak{d}=\mathfrak{d}_{S(X)}$ in \propref{mainprop} that the spectral dimension of $X$ always satisfies the inequality $\mathfrak{d}\geq 1$.  
\end{rmk}

Since $X$ is a compact metric space in the geodesic metric, $d_{geo}$, we define its associated Lipschitz seminorm $\lip_{g}$ as in \eqref{lip1};  namely,

\[\lip_{g}(f)=\sup\left\{\frac{|f(x)-f(y)|}{d_{geo}(x,y)} : x\neq y\right\}.\]
  
\noindent The following lemma will be useful in recovering $d_{geo}$ from the Dirac operator via Formula 1:

\begin{lem} \label{lemma1}
Let $X$ be a compact length space satisfying Axioms 1 and 2, and let {\normalfont{$\lip_{g}$}} be the Lipschitz seminorm for the compact metric space $X$ with respect to $d_{geo}$.  Then, for any function $f$ in the domain of $D$, we have

\[||Df||_{\infty,X}={\normalfont{\lip_{g}}}(f).\]
\end{lem}

\vspace{1mm}

\begin{proof}
For any $f$ in the domain of $D$, we have (with $i\vcentcolon=\sqrt{-1}$)

\[||Df||_{\infty,X}= \sup_{j}\left\{||D_{j}f||_{\infty,R_{j}}\right\}=\sup_{j}\left\{\left|\left|\frac{1}{i}\frac{df}{dx}\right|\right|_{\infty,R_{j}}\right\}\]

\[=\sup_{j}\left\{\sup_{p,q\in R_{j}}\left\{\frac{|f(p)-f(q)|}{d_{geo}(p,q)}\right\}\right\}\leq \lip_{g}(f).\]

\noindent The first equality follows since $\mathcal{R}$ is dense in $X$, according to Axiom 1.  The last inequality is clear since $\lip_{g}$ is the supremum over all $p\neq q\in X$, not just those $p\neq q$ restricted to being in the same $R_{j}$. \\

The inequality in the other direction will come from Axiom 2.  First suppose $p\in \mathcal{B}$ and $q\in X$.  Then there is a geodesic from $p$ to $q$ which is a concatenation of $R_{j}$ curves.  Let $\{(p_{j},p_{j+1})\}$ be the sequence of pairs of endpoints tracking the $R_{j}$ curves such that $p_{1}=p$ and $\lim_{n\rightarrow \infty} p_{n}=q$ along $\gamma$.  We have the following estimate:

\[|f(p)-f(p_{n})|=|\sum_{j=1}^{n}f(p_{j})-f(p_{j+1})|\leq \sum_{j=1}^{n}|f(p_{j})-f(p_{j+1})|\]

\[\leq \sum_{j=1}^{n}\left(d_{geo}(p_{j},p_{j+1})||D_{j}f||_{\infty,R_{j}}\right)\leq \left(||Df||_{\infty, X}\right)\sum_{j=1}^{n}d_{geo}(p_{j},p_{j+1})\]

\[=||Df||_{\infty, X}d_{geo}(p,p_{n}).\]
 
\noindent  By the continuity of $f(x)$ and letting $a(x)\vcentcolon= d_{geo}(p,x)$, we deduce that

\[\frac{|f(p)-f(q)|}{d_{geo}(p,q)}\leq ||Df||_{\infty, X}.\]
 
\noindent  Note that the above estimate does not rely on the fact that $p\in \mathcal{B}$, but only on the fact that $p$ is an endpoint; see \rmkref{rmk2} above.\\

   Now suppose $p$ and $q$ are arbitrary distinct points in $X$.  By Axiom 2, there is a minimizing geodesic $\gamma$ in $X$ connecting $p$ and $q$.  Since $\mathcal{B}$ is dense in $X$,  $\gamma$ intersects some point of $\mathcal{B}$, say $r_{0}$.  Let $l_{1}$ be the length of $\gamma$ from $r_{0}$ to $p$ and $l_{2}$ be the length of $\gamma$ from $r_{0}$ to $q$.  Thus the total length of $\gamma$ is $l_{1}+l_{2}$. \\

    By Axiom 2, there exist minimizing geodesics $\gamma_{1}$ and $\gamma_{2}$ from $r_{0}$ to $p$, and from $r_{0}$ to $q$, respectively, which are countable concatenations of curves in $\mathcal{R}$ originating out of $r_{0}$.  It follows that $\gamma_{1}$ has length $l_{1}$ and $\gamma_{2}$ has length $l_{2}$.  For completeness, we briefly explain why this is the case.  Indeed, supposing the length of $\gamma_{1}$ is less than $l_{1}$ implies that the concatenation of $\gamma_{1}$ with $\gamma$ from $r_{0}$ to $q$ would have length less than $l_{1}+l_{2}$, contradicting the fact that $\gamma$ is a shortest path.  Moreover, if we suppose that $\gamma_{1}$ has length greater than $l_{1}$, then it follows that $\gamma$ is a shorter path from $r_{0}$ to $p$ , contradicting the fact that $\gamma_{1}$ is a shortest path.  The same arguments hold for $\gamma_{2}$.  Hence, the concatenation of $\gamma_{1}$ and $\gamma_{2}$ has length $l_{1}+l_{2}$ and is therefore a geodesic between $p$ and $q$.  \\
    
     Let $\gamma_{2}$ be tracked by endpoints $\{r_{i}\}$ of the concatenated curves $R_{i}$ in $\mathcal{R}$ such that $r_{1}=r_{0}$ and $\lim_{i\rightarrow \infty} (r_{i})=q$.  Define $\gamma_{1i}$ to be the path obtained by concatenating the first $i$ paths of $\gamma_{2}$ with $\gamma_{1}$ at $r_{0}$.  Using the estimate for an endpoint to a point in $X$, applied to $\gamma_{1i}$ from $r_{i}$ to $p$, we have 

\[\frac{|f(p)-f(r_{i})|}{d_{geo}(p,r_{i})}\leq ||Df||_{\infty, X} \mbox{  for all  } i\in \mathbb N.\]

\noindent  Again using the continuity of the functions $f(x)$ and $a(x)\vcentcolon= d_{geo}(p,x)$, we have 

\[\frac{|f(p)-f(q)|}{d_{geo}(p,q)}\leq ||Df||_{\infty, X}.\]

\noindent Since $p$ and $q$ are arbitrary points in $X$, it follows that $\lip_{g}(f)\leq ||Df||_{\infty, X}$, as desired.  
\end{proof}

We can now state and prove our main result for this section:

\begin{thm} \label{mainthm}
Let $X$ be a compact length space satisfying Axioms 1 and 2, and let $d_{X}$ be the distance function on $X$ induced by the spectral triple via Formula 1.  Then  $d_{X}=d_{geo}$, with the spectrum $\sigma(D)$ of the Dirac operator and the spectral dimension $\mathfrak{d}=\mathfrak{d}_{S(X)}$ as given in \propref{mainprop}.
\end{thm}

\begin{proof}
First, we note that the spectrum of the Dirac operator and the spectral dimension are given as in \propref{mainprop} since $X$ satisfies Axiom 1.\\

To prove that $d_{X}=d_{geo}$, let  $p,q\in X$.  To compare $d_{geo}(p,q)$ with $d_{X}(p,q)$, note that for any $f$ (in the domain of $D$) such that $||Df||_{\infty, X}\leq 1$, we have by \lemref{lemma1} that $\lip_{g}(f)=||Df||_{\infty, X}$ and hence,

\[\frac{|f(p)-f(q)|}{d_{geo}(p,q)}\leq 1 .\]

\noindent In this case, $|f(p)-f(q)|\leq d_{geo}(p,q)$, and it holds that $d_{X}(p,q)\leq d_{geo}(p,q)$.  To get the inequality in the other direction, define the function $h(x)=d_{geo}(p,x)$.  Then, $\lip_{g}(h)=1$ and 

\[|h(p)-h(q)|=|0-d_{geo}(p,q)|=d_{geo}(p,q). \]

\noindent Therefore, since $h$ is a Lipschitz function on $X$, $h$ is witness to the inequality $d_{geo}(p,q)\leq d_{X}(p,q)$, and we have shown that $d_{X}(p,q)=d_{geo}(p,q)$, as desired. 
\end{proof}

\thmref{mainthm} is an extension of Connes' theorem on a compact Riemannian manifold to the class of compact length spaces determined by Axioms 1 and 2.  In the next two sections, we provide examples of fractal sets which fall in this class of length spaces.  The first example is the Sierpinski gasket, in which case its geometry has been recovered using similar methods in \cite{CrIvLap}.  The second example is the harmonic gasket and its measurable Riemannian geometry which provide a setting closer to that of Riemmanian manifolds, and for which our results are new.

\section{Spectral Geometry of the Sierpinski Gasket}

In this section, we show that the Sierpinski gasket, $K$, is a model for \thmref{mainthm}.  It is shown in \cite{CrIvLap} that $K$ is a compact length space.  It remains to prove that $K$ satisfies Axiom 1 and Axiom 2:

\begin{prop} \label{KA12}
The Sierpinski gasket $K$ satisfies Axioms 1 and 2.
\end{prop}

\begin{proof}
Let $K$ be decomposed into its cell edges by decomposing each of its graph cells $\Gamma_{w}$ into $\Gamma_{w,j}$, for $j\in \{l,r,b\}$, where $l$, $r$, and $b$ denote the left, right, and bottom, respectively, of each graph cell (triangle) of the gasket. Then, the union over $|w|=n\in\mathbb{N}$ and $j\in\{l,r,b\}$ of the $\Gamma_{w,j}$'s is the countable union of cell edges whose closure is $K$.  Indeed, this union contains the set of vertices $V^{*}$, which is dense in $K$.  We can reorder the cell edges with $\mathbb{N}$, with each cell edge given by $R_{j}$, for some $j\in\mathbb{N}$, in non-increasing order.  Let
  
\[\mathcal{R}=\bigcup_{j\in \mathbb{N}}R_{j}.\]
 
\noindent Then $K=\overline{\mathcal{R}}$.  The first graph cell has $R_{1}$, $R_{2}$, and $R_{3}$ as its edges, which are of equal finite length.  An $R_{j}$ curve which is an edge of a graph cell of $\Gamma_{m}$ has length proportional to $(1/2^{m})$.  There are $3^{m}$ curves of this length.  It follows that the sequence of (Euclidean) lengths $\alpha_j=L(R_j)$ satisfies $L(R_{j})\rightarrow 0$ as $j\rightarrow \infty$ and that each $R_{j}$ is a rectifiable $C^{1}$ curve (a straight line segment in $\R^{2}$ with bounded length).  Therefore, $K$ satisfies Axiom 1.\\

We now show that $K$ satisfies Axiom 2.  The key properties which allow $K$ to satisfy Axiom 2 are its connectedness and the fact that every edge curve is itself a minimizing geodesic between its endpoints.  Let $p\in V^{*}$ and $q\in K$.  A shortest path to $q$ from $p$ is constructed by considering the lowest graph approximation $\Gamma_{m}$ which puts $p$ and $q$ in separate cells, $K_{w}$ and $K_{w'}$, respectively, with $|w|=|w'|=m$. First suppose $p$ is a vertex in the graph cell $\Gamma_{w}$, with $|w|=m$. \\

 By connectedness, the shortest (in fact, any) path from $p$ to $q$ must pass through a vertex $v$ of $\Gamma_{w'}$.   There is an $R_{j}$ which is an edge of a graph cell in $\Gamma_{m}$ connecting $p$ to $v$.  Each $R_{j}$ is a straight line segment and is therefore itself a minimizing geodesic between its endpoints.  Hence, the curve $R_{j}$ connecting $p$ to $v$ suffices as the \textsl{first leg} of the shortest path from $p$ to $q$.  We repeat the previous argument, finding the lowest graph approximation placing $v$ and $q$ in different cells.  Since $v$ is necessarily a vertex of this (higher) graph approximation, we apply the same argument to the cells separating $v$ and $q$.  Continuing in this manner, we obtain a path that is a countable concatenation of $R_{j}$'s which are edges of cells whose diameters go to zero.  The finite intersection property yields a unique limit point, which is necessarily $q$.  \\

 For the case when $p\in V^{*}$ but is a vertex of a higher approximation than $\Gamma_{m}$, we can use the special case above. Let $u$ be a vertex of $\Gamma_{w}$ and $p\in K_{w}$.  The argument above for a shortest path from $p$ to $q$ applies to the shortest path from $u$ to $p$.  However, in this case, since $p\in V^{*}$, the process terminates after finitely many iterations.  Indeed, there is a `last' graph cell the path must travel to until it is at most one edge curve away from $p$.  This finite concatenation can be reversed from $p$ to $u$ and then concatenated with the path from $u$ to $q$.  The resulting path is a minimizing geodesic from $p$ to $q$ which is a countable concatenation of $R_{j}$ curves.  Since $V^{*}$ is dense in $K$,  Axiom 2 is satisfied.
\end{proof}

In the light of \propref{KA12}, we have the following immediate corollary to \thmref{mainthm}.

\begin{cor}  \label{corK}
The spectral triple, $S(K)$, constructed from the countable sum of $R_{j}$-curve triples, satisfies the following{\normalfont{:}}  {\normalfont{(}}1{\normalfont{)}} The distance function induced by $S(K)$ via Formula 1 coincides with the geodesic distance function on $K${\normalfont{;}} {\normalfont{(}}2{\normalfont{)}}  The spectral dimension of $S(K)$ is equal to $\log 3/\log2$.\footnotemark[1] \\
\end{cor}

\footnotetext[1]{This value coincides with the Hausdorff dimension of $K$, both with respect to the Euclidean metric (as is well known, see, e.g., \cite{Falconer, Mattila}) and with respect to the geodesic metric of $K$ (according to the results of \cite{CrIvLap}).  It does not, however, coincide with the Hausdorff metric of $\KH$ with respect to the geodesic metric, which is also $>1$ but close to 1.3 (as was recently shown in \cite{Kaj1, Kaj2, Kaj3}).  (To our knowledge, the fractal dimension of $\KH$ with respect to the Euclidean metric is still unknown.)}

\section{Spectral and Measurable Riemannian Geometry}

  As mentioned in Subsection 2.4, it is shown in \cite{Kig4} that $K_{H}$ is a compact length space and that the minimizing geodesics have a representation in the language of measurable Riemannian geometry analogous to the corresponding representation of geodesics in Riemannian geometry.  
In this section, we show that the harmonic gasket, $K_{H}$, satisfies Axioms 1 and 2 and is thus a model for \thmref{mainthm}.  The result is that we are able to recover Kigami's geodesic distance from the spectral triple and Dirac operator for the harmonic gasket via Formula 1. Let $D$ be the 
Dirac operator on $\KH$ and let $\mathcal{A}=C(\KH)$.  \corref{maincor} at the end of this section yields the following result:

\begin{thm} \label{Kigami}  Let $p$ and $q$ be arbitrary points in $\KH$, and let $\gamma$ be a minimizing geodesic from $p$ to $q$ such that $\gamma(t_{1})=p$ and $\gamma(t_{2})=q$.  Then we have
\begin{equation}  \label{intgeo}
\int_{t_{1}}^{t_{2}}{\left\langle \dot{\gamma},Z\dot{\gamma}\right\rangle}^{\frac{1}{2}}dt=\sup_{a\in \mathcal A}\left\{|a(p)-a(q)| : ||[D,\pi a]||\leq 1 \right\}. 
\end{equation}
\end{thm}

\vspace{2mm}

 As measurable Riemmannian geometry extends notions of smooth Riemmanian geometry to a certain fractal set, equality \eqref{intgeo} extends Connes' theorem for a compact Riemannian manifold to this setting.  
We first show that $K_{H}$ satisfies Axioms 1 and 2:

\begin{prop}  \label{A12}
The harmonic gasket $K_{H}$ satisfies Axioms 1 and 2.
\end{prop}

\begin{proof}
Using the homeomorphism $\Phi$ between $K$ and $\KH$, and whose definition was recalled towards the end of Subsection 2.3, we can decompose $K_{H}$ from the decomposition we used for $K$.  The edges of graph cells in $K_{H}$ are exactly given as $\Phi(R_{j})$, where the $R_{j}$'s are edges of graph cells of $K$.  Let
  
\[\mathcal{R}=\bigcup_{j\in \mathbb{N}}\Phi(R_{j}).\]
 
\noindent Because $\Phi$ is a homeomorphism and since $K$ satisfies Axiom 1 (by \propref{KA12}), we have that $K_{H}=\overline{\mathcal{R}}$.   By Theorem 5.4 in \cite{Kig4} (Theorem 4.7 in \cite{Tep1} gives the same result),  $\Phi(R_{j})$ is a $C^{1}$ curve.  Moreover, by Lemma 5.5 in \cite{Kig4}, the curve $\Phi(R_{j})$ is rectifiable.  Since every cell edge is an affine image of an edge of the first graph cell, with maximum eigenvalue $3/5$, it follows that the sequence of (Euclidean) lengths of the curve $\Phi(R_j)$ satisfies $L(\Phi(R_{j}))\rightarrow 0$ as $j\rightarrow \infty$.  Therefore, $K_{H}$ satisfies Axiom 1.  \\

  The argument that $K_{H}$ satisfies Axiom 2 is analogous to the argument for $K$ (given in the second part of the proof of \propref{KA12}), except that convexity is a proxy for straight lines.  To be precise, we need to show that if $p$ and $q$ are the endpoints of an edge, $\Phi(R_{j})$, then $\Phi(R_{j})$ is the minimizing geodesic in $K_{H}$ between $p$ and $q$ (and thus the shortest path between any two points on $R_{j}$ lies on $\Phi(R_{j})$).  Let $p$ and $q$ be vertices of a cell $K_{H,w}$ of $K_{H}$ and $\overline{pq}$ be the straight line segment in $M_{0}$ connecting $p$ and $q$.  Let $\Phi(R_{j})$ be the cell edge connecting $p$ and $q$ and  $D_{pq}$ be the compact region bounded by $\overline{pq}\cup \Phi(R_{j})$.  Lemma 5.5 in \cite{Kig4} states that $D_{pq}$ is convex and that $\Phi(R_{j})$ is rectifiable. \\
  
   Theorem 5.2 in \cite{Kig4} states that if $C\subset D$ are compact subsets in $\R^{2}$ with $C$ convex and $\partial D$ a rectifiable Jordan curve, then $L(\partial C)\leq L(\partial D)$.  Lemma 5.6 in \cite{Kig4} uses this theorem to show that $\Phi(R_{j})$ is a shortest path between $p$ and $q$ among all rectifiable paths in $K_{H,w}$ between $p$ and $q$.  Since we would like to show this holds among all rectifiable paths in $K_{H}$, we follow the proof of Lemma 5.6 in \cite{Kig4}, except that we allow for $\widetilde{pq}$ to be any rectifiable (w.l.o.g., non-intersecting) curve in $K_{H}$ connecting $p$ and $q$. \\
   
   Let $D'_{pq}$ be the compact region bounded by $\widetilde{pq}\cup \overline{pq}$.  Since $\Phi$ is a homeomorphism and thus preserves the holes, and hence the interior and exterior of $K$, we have that $(D_{pq}\backslash \Phi(R_{j}))\cap K_{H}$ is empty.  It therefore holds that $D_{pq}\subset D'_{pq}$ and by Theorem 5.2 in \cite{Kig4}, $L(\widehat{pq}\cup \overline{pq})\leq L(\widetilde{pq}\cup \overline{pq})$.  Subtracting off the segment, $\overline{pq}$, which the two boundaries have in common, yields  $L(\widehat{pq})\leq L(\widetilde{pq})$.  We have now shown that $\Phi(R_{j})$ is the minimizing geodesic in $K_{H}$ between $p$ and $q$.\\ 

   Next, let $p\in \Phi(V^{*})$ (i.e., $p$ is a vertex of $K_{H}$) and let $q\in K_{H}$.  Since $K_{H}$ is topologically equivalent to $K$, the argument for a geodesic from $p$ to $q$ is the same as for $K$ (in the proof of \propref{KA12}), except that the straight line edges, $R_{j}$, are replaced with the harmonic edges, $\Phi(R_{j})$.  Therefore, a geodesic from $p$ to $q$ can be given as a countable concatenation of $\Phi(R_{j})$'s.  Since $\Phi(V^{*})$ is dense in $K_{H}$ (because $V^*$ is dense in $K$ and $\Phi$ is a homeomorphism from $K$ onto $\KH$), it follows that $K_{H}$ satisfies Axiom 2.
\end{proof}

\propref{A12} shows that $K_{H}$ is a model for \thmref{mainthm}, and thus we have the following corollary, which is the exact counterpart for $\KH$ of \corref{corK} stated for $K$ at the end of Section 4:

\begin{cor}  \label{maincor}
The spectral triple, $S(K_{H})$, constructed from the countable sum of \hspace{.1mm} $\Phi(R_{j})$-curve triples satisfies the following{\normalfont{:}} {\normalfont{(}}1{\normalfont{)}}  The spectral distance induced by $S(K_{H})$ via Formula 1 coincides with Kigami's geodesic distance on $K_{H}${\normalfont{;}} {\normalfont{(}}2{\normalfont{)}} The spectrum $\sigma(D)$ of the Dirac operator and the spectral dimension $\mathfrak{d}=\mathfrak{d}_{S(K_{H})}$ are given as in \propref{mainprop} {\normalfont{(}}with $X=K_{H}${\normalfont{)}}.
\end{cor}

\section{Alternate Constructions for $K_{H}$}

In this section, we first construct the Dirac operator for $K_{H}$ in analogy with the construction for $K$ in \cite{CrIvLap}.  More precisely, we construct a spectral triple for $\KH$ using triples for the graph cells (distorted triangles) of the harmonic gasket, and therefore the construction comes directly from circle triples.  This construction has the benefit of keeping track of the `holes' in the gasket.  We show that it also recovers Kigami's geometry, yet the spectrum of the Dirac operator, though asymptotically the same, is not exactly the same as in the edge construction in the previous section.  We conclude this section with a construction which is the direct sum of the edge construction and the cell construction.  It is shown that this construction also recovers Kigami's measurable Riemannian geometry.

\subsection{Harmonic cell triple}

Recall from Subsection 2.3 that $\Gamma_{w}$ denotes a graph cell of $K$ associated with the finite word $w$.  Using the homeomorphism $\Phi$ from $K$ onto $\KH$, we can define the corresponding graph cell $T_{w}=\Phi(\Gamma_{w})$;  clearly, $T_{w}$ is a graph cell of $K_{H}$.  We can construct a triple on $ T_{w}$ by carrying the spectral triple on a circle directly to $ T_{w}$, as is done in \cite{CrIvLap} for an arbitrary graph cell of the Sierpinski gasket.  Let $r$ be the radius of a circle.  Since it is the complex continuous functions on the circle that are of interest, we make the natural identification with the complex continuous $2\pi r$-periodic functions on the real line.  Let the $\R^{2}$ induced arclength of $ T_{w}$ be denoted by $\alpha_{w}$.  (Here and in the sequel, we use the notation analogous to the one introduced towards the end of Subsection 2.2.) \\

 Considering a circle of radius $\alpha_{w}$, the appropriate algebra of functions consists of the complex continuous $2\pi \alpha_{w}$-periodic functions on the real line.  Let $r_{w}: [-\pi \alpha_{w},\pi \alpha_{w}]\rightarrow  T_{w}$ be an arclength parameterization of $ T_{w}$, counterclockwise, with $r_{w}(0)$ equal to the vertex joining the bottom and right sides of $ T_{w}$.  According to Definition 8.1 in \cite{CrIvLap}, the mapping $r_{w}$ induces a surjective homomorphism $\Psi_{w}$ of $C(\KH)$ onto $C([-\pi \alpha_{w},\pi \alpha_{w}])$ given by 

\[\Psi_{w}(f)(\tau)\vcentcolon=f(r_{w}(\tau)), \]

\noindent for $f\in C(\KH)$ and $\tau\in [-\pi \alpha_{w},\pi \alpha_{w}]$.  Let

\[\mathfrak{H}_{w}=L^{2}([-\pi \alpha_{w},\pi \alpha_{w}], (1/2\pi \alpha_{w})m),\] 

\noindent where $m$ is the Lebesgue measure on $[-\pi \alpha_{w},\pi \alpha_{w}]$, and let $\Pi_{w}:C(\KH)\rightarrow B(\mathfrak{H}_{w})$ be the representation of $f$ in $C(\KH)$ defined as the multiplication operator which multiplies by $\Psi_{w}(f)$.  We will again use the translated Dirac operator and define $\mathfrak{D}_{w}=D_{\alpha w}^{t}$.  (See Subsection 2.2 above.) \\

The triple $\mathfrak{S}(T_{w})=(C(\KH),\mathfrak{H}_{w}, \mathfrak{D}_{w})$ is an unbounded Fredholm module with representation $\Pi_{w}$.   The results in the following proposition follow from the corresponding results regarding the spectral triple on a circle obtained in Section 2 of \cite{CrIvLap}.  

\begin{prop} \label{hcellprop}
The triple $\mathfrak{S}( T_{w})=(C(\KH),\mathfrak{H}_{w}, \mathfrak{D}_{w})$ associated to $ T_{w}$ is an unbounded Fredholm module satisfying the following properties{\normalfont{:}}

\begin{enumerate}

\item The spectrum of the Dirac operator, $\mathfrak{D}_{w}$, is given by \[\sigma(\mathfrak{D}_{w})=\left\{\left[\left(\frac{(2k+1)\pi}{2\alpha_{w}}\right)\right] : k\in \mathbb Z \right\}. \] 

\item The metric $d_{w}$ induced by $\mathfrak{S}( T_w)$ on $ T_{w}$ coincides with the $\R^{2}$ induced arclength metric $l_{ T w}$ on $ T_{w}$.

\item The spectral dimension of $T_{w}$ is $1$.
\end{enumerate}
\end{prop}

\begin{rmk}  \label{rmk4}
\propref{hcellprop} does not state that the metric $d_{w}$ coincides with the restriction to the graph cell $\Tw$ of Kigami's geodesic metric on $\KH$, because in general this is not the case.  Indeed, points on different sides of $T_{w}$ will be connected by a geodesic that does not lie completely on $T_{w}$, and thus $d_{w}\geq d_{geo}$.  However, from the proof of \propref{A12}, it follows that $d_{w}$ restricted to an edge of $T_{w}$ coincides with Kigami's geodesic distance.
\end{rmk}

\subsection{Construction from cell triples}

We now construct a spectral triple on $\KH$ using the countable sum of triples  $\mathfrak{S}( T_{w})=(C(\KH),\mathfrak{H}_{w}, \mathfrak{D}_{w}).$  This is a natural construction of the spectral triple with respect to its \textsl{holes and connectedness}; this is also the construction used for the Sierpinski gasket in \cite{CrIvLap}. \\

To be precise, this construction yields a triple for each closed path, or \textsl{cycle}, in the space.  Following the line of reasoning on page 23 of \cite{CrIvLap}, each of these triples associated to a cycle induces an element in the $K$-homology of each graph approximation of $\KH$.  Each of these members of the $K$-homology group measures the winding number of a nonzero continuous function around the cycle to which it is associated, keeping track of the connectedness type of the graph approximation.\\

To formally construct the countable sum of $\mathfrak{S}( T_{w})$ triples, we will use the following notation:

\begin{enumerate}

\item $\mathfrak{H}_{\KH}=\bigoplus_{|w|=n}^{n\in \mathbb N}\mathfrak H_{w}$;

\item $\Pi_{\KH}=\bigoplus_{|w|=n}^{n\in \mathbb N}\Pi_{w}$;

\item $\mathfrak{D}_{\KH}=\bigoplus_{|w|=n}^{n\in \mathbb N}\mathfrak{D}_{w}.$

\end{enumerate}

\vspace{2mm}

\noindent In each case, the countable orthogonal direct sum is extended over 

\[W^{*}= \bigcup_{m\geq 0} W_m,\]

\noindent the set of all finite words, where $W_m$ is the set of all words $w$ of length $|w|=m\in \N$  on the alphabet $S={1,2,3}$;  see Subsection 2.3 above.  

Furthermore, the countable sum of the $\mathfrak{S}( T_w)$ triples is defined as $\mathfrak{S}(\KH)=(C(\KH),\mathfrak{H}_{\KH},\mathfrak{D}_{\KH})$.  In order to show that  $\mathfrak{S}(\KH)$ is a spectral triple, we first note that a function in the image of $\Pi_{\KH}$ is densely defined on $\KH$, so that we indeed have a faithful representation.  \\

Next, we show that there is a dense set of functions $f$ in $C(\KH)$ such that the commutator of $\Pi_{\KH}(f)$ with the Dirac operator $\mathfrak{D}_{\KH}$ is bounded.  The real-valued linear functions on $\R^{2}$ restricted to $\KH$, are dense in $C(\KH)$.  Furthermore, any real-valued linear function, $f(x,y)=ax+by$, restricted to the graph cell $\mathcal{T}_{w}$, has a bounded commutator with $\mathfrak{D}_{w}$ with bound $|a|+|b|$, independent of $w$. Thus $||[\mathfrak{D}_{\KH},\Pi_{\KH}(f)]||\leq |a|+|b|$ and hence the real-valued linear functions on $\R^{2}$, restricted to $\KH$, form a dense subset of $C(\KH)$   consisting of elements having bounded commutators with $\mathfrak{D}_{\KH}$.  \\

To see that the operator $(\mathfrak{D}_{\KH}^{2}+I)^{-1}$ is compact, we look at the eigenvalues of $\mathfrak{D}_{\KH}$, which are given by the disjoint union of the eigenvalues of all of the $\mathfrak{D}_{w}$'s:

\[\sigma(\mathfrak{D}_{\KH})=\bigcup_{n\in \N}\bigcup_{|w|=n}\left\{\left[\frac{(2k+1)\pi}{2\alpha_{w}}\right] : k\in \mathbb Z \right\},\] 

\noindent where we have used part 1 of \propref{hcellprop}.  As mentioned before, the $\alpha_{w}$'s are the lengths of the boundaries of the $w$-cells.  As a result, the eigenvalues of $(\mathfrak{D}_{\KH}^{2}+I)^{-1}$ go to zero and therefore, $(\mathfrak{D}_{\KH}^{2}+I)^{-1}$ is compact.  In addition, one verifies that $\mathfrak{D}_{\KH}$ is symmetric when acting on its eigenvectors, so that it is 
self-adjoint.\\

To compare the spectral distance function induced by $\mathfrak{S}(\KH)$ with $d_{geo}$, we have the following analog of \lemref{lemma1} in Section 3, which characterizes $||\mathfrak{D}_{\KH}||_{\infty, \KH}$ in terms of $d_{geo}$.

\begin{lem} \label{lemma2}
For any function $f$ in the domain of $\mathfrak{D}_{\KH}$, we have

\[||\mathfrak{D}_{\KH}f||_{\infty,\KH}={\normalfont{\lip_{g}}}(f).\]
\end{lem}

\vspace{1mm}

\begin{proof} 
For any $f$ in the domain of $\mathfrak{D}_{\KH}$, we have (with $i\vcentcolon=\sqrt{-1}$)

\[||\mathfrak{D}_{\KH}f||_{\infty,\KH}=\sup_{w}\left\{||\mathfrak{D}_{w}f||_{\infty, T w}\right\}=\sup_{w}\left\{\left|\left|\frac{1}{i}\frac{df}{dx}\right|\right|_{\infty, T w}\right\}\]

\[=\sup_{w}\left\{\sup_{p,q\in  T w}\left\{\frac{|f(p)-f(q)|}{d_{w}(p,q)}\right\}\right\}\leq \sup_{w}\left\{\sup_{p,q\in  T w}\left\{\frac{|f(p)-f(q)|}{d_{geo}(p,q)}\right\}\right\} \]

\[\leq \lip_{g}(f),\]

\noindent since $d_{w}\geq d_{geo}$ (as was noted in \rmkref{rmk4}).  The last inequality holds since $\lip_{g}(f)$ is the supremum over all possible non-diagonal pairs of points on the harmonic gasket, which includes the non-diagonal pairs of points restricted to belonging to the same graph cell. \\

 To achieve the reverse inequality, we note that the critical inequality used to get this direction in \lemref{lemma1} was

\[|f(p_{j})-f(p_{j+1})|\leq d_{geo}(p_{j},p_{j+1})||D_{R_{j}}f||_{\infty,R_{j}},\]

\noindent where the $p_{j}$'s represent the decomposition of the geodesic constructed in \lemref{lemma1}, and $R_{j}$ is the edge curve connecting $p_{j}$ to $p_{j+1}$.  Recalling \rmkref{rmk4} following \propref{hcellprop}, we have that the spectral distance induced on 
$T_{w}$ by $\mathfrak{S}(T_w)$ coincides with $d_{geo}$ when restricted to the edges of $\Tw$.  This is of course a sufficient condition to replace $R_{j}$ with $T_{w}$ in the above inequality.   Indeed, for $p_{j}$ and $p_{j+1}$ belonging to the same edge, 

\[\frac{|f(p_{j})-f(p_{j+1})|}{d_{geo}(p_{j},p_{j+1})}=\frac{|f(p_{j})-f(p_{j+1})|}{d_{w}(p_{j},p_{j+1})}\leq ||\mathfrak{D}_{w}f||_{\infty, T_w}.\]

\noindent Therefore, 

\[|f(p_{j})-f(p_{j+1})|\leq d_{geo}(p_{j},p_{j+1})||\mathfrak{D}_{w}f||_{\infty, T_w}.\]

\noindent Now it follows from the argument used in the proof of \lemref{lemma1} that for an arbitrary point $q$ and a vertex $p$,

\[\frac{|f(p)-f(q)|}{d_{geo}(p,q)}\leq ||\mathfrak{D}_{\KH}f||_{\infty,\KH}.\]

\noindent The extension to the case when $p$ and $q$ are both arbitrary points in $\KH$ also follows the same argument as in the proof of \lemref{lemma1} and therefore,

\[\lip_{g}(f)\leq ||\mathfrak{D}_{\KH}f||_{\infty,\KH}.\]

\noindent We deduce that $||\mathfrak{D}_{\KH}f||_{\infty,\KH} = \lip_{g}(f)$, and hence, the proof of the lemma is completed. 
\end{proof}

Let $h_{spec}$ be the distance function induced by $\mathfrak{S}(\KH)$ and let $\mathfrak{b}_{\KH}$ be the spectral dimension of $\KH$ with respect to $\mathfrak{S}_{\KH}$.  Just as \lemref{lemma1} gives $d_{spec}=d_{geo}$, \lemref{lemma2} gives $h_{spec}=d_{geo}$ using the exact same argument as in the proof of \thmref{mainthm}.  The following theorem, an analog for the harmonic gasket of \propref{mainprop} and \thmref{mainthm}, summarizes the results for the spectral triple $\mathfrak{S}(\KH)=(C(\KH),\mathfrak{H}_{\KH},\mathfrak{D}_{\KH})$.

\begin{thm} \label{altmainthm}
The triple $\mathfrak{S}(\KH)=(C(\KH),\mathfrak{H}_{\KH},\mathfrak{D}_{\KH})$ associated to $\KH$ is a spectral triple satisfying the following properties{\normalfont{:}}

\begin{enumerate}

\item The spectrum of the Dirac operator, $\mathfrak{D}_{\KH}$, is given by

 \[\sigma(\mathfrak{D}_{\KH})=\bigcup_{n\in \N}\bigcup_{|w|=n} \left\{\left[\frac{(2k+1)\pi}{2\alpha_{w}}\right] : k\in \mathbb Z\right\}.\] 

\item The metric distance $h_{spec}$ induced by $\mathfrak{S}(\KH)$ coincides with Kigami's geodesic distance, $d_{geo}$.

\item The spectral dimension $\mathfrak{b}_{\KH}$ is the infimum of all $p>1$ such that 

\[\sum_{n\in \N}\sum_{|w|=n} (\alpha_{w})^p <\infty .\]

In particular, $\mathfrak{b}_{\KH}\geq 1$.

\end{enumerate}
\end{thm}

\begin{proof}  That $\mathfrak{S}(\KH)$ is a spectral triple for $\KH$ and the first two claims (1 and 2) of \thmref{altmainthm} follow directly from the text just above \thmref{altmainthm}.  The third claim (3), much as in the proof of \propref{mainprop}, follows from the fact that by definition, and in light of the first part (1), 

\[\mathfrak{b}_{\KH}=\inf\left\{p>0: \sum_{n\in \N}\sum_{|w|=n}\sum_{k\in \mathbb Z} \left|\frac{(2k+1)\pi}{2 \alpha_{w}}\right|^{-p} < \infty\right\}.\]

\noindent It is clear that the triple sum in the expression above is finite if and only if the double sum over $n$ and $w$, $\sum_{n\in \N}\sum_{|w|=n} (\alpha_{w})^p$, and the sum over $k$, $\sum_{k\in \N} |2k +1|^{-p}$, are both finite.  Since, clearly, 
$\sum_{k\in \N} (2k+1)^{-p}<\infty$ if and only if $p>1$, it follows that 

\[\mathfrak{b}_{\KH}=\inf\left\{p>1: \sum_{n\in \N}\sum_{|w|=n} (\alpha_{w})^p <\infty\right\},\]

\noindent as desired.  
\end{proof}

The corollary to follow compares the geometries of the Sierpinski gasket induced by $S(\KH)$ and $\mathfrak{S}(\KH)$:

\begin{cor} \label{compcor}
For the spectral distance functions, $d_{spec}$ and $h_{spec}$, and the spectral dimensions, $\mathfrak{d}_{\KH}$ and $\mathfrak{b}_{\KH}$, the following equalities hold{\normalfont{:}}

\begin{enumerate}

\item $d_{spec}=h_{spec}$.

\item $\mathfrak{d}_{\KH}=\mathfrak{b}_{\KH}$.

\end{enumerate}
\end{cor}

\begin{proof}  
The first fact follows immediately from \corref{maincor} and \thmref{altmainthm}.  The second fact is true since 
\[\alpha_{w}=\sum_{s\in \{L,R,B\}}{\alpha_{w,s}}.\]
\end{proof}

\subsection{The direct sum of $S(\KH)$ and $\mathfrak{S}(\KH)$}
 
In this section, we point out that the two spectral triples on $\KH$, $S(\KH)$ and $\mathfrak{S}(\KH)$, constructed above can be summed together, giving a spectral triple that also recovers Kigami's distance on $\KH$.  This construction involves the refinement of the curve triple construction and also keeps track of the holes in $\KH$.\\

\begin{thm}
Let $S(\oplus)=S(\KH) \oplus \mathfrak{S}(\KH)$, with $\pi_{\oplus}=\pi_{\KH} \oplus \Pi_{\KH}$.  Then $S(\oplus)$ is a spectral triple for $K_{H}$ and the  distance, $d_{\oplus}$, induced by $S(\oplus)$ on $K_{H}$ coincides with Kigami's geodesic distance on $\KH$.  
\end{thm}

\begin{proof}
Let $D_{\KH}$ denote the Dirac operator associated to $S(\KH)$.  It is clear from \propref{mainprop} and \thmref{altmainthm} that for any real-valued linear function, $f(x,y)=ax+by$ on $K_{H}$, we have

\[||[D_{\KH},\pi_{\KH}(f)]||\leq |a|+|b| \mbox{  and  } ||[\mathfrak{D}_{\KH},\Pi_{\KH}(f)]||\leq |a|+|b|. \]

\noindent  Since $D_{\oplus}=D_{\KH}\oplus \mathfrak{D}_{\KH}$, it follows that

\[||[D_{\oplus},\pi_{\oplus}(f)]||\leq |a|+|b|.\]

\noindent (Recall that the underlying Hilbert space is the orthogonal direct sum of the Hilbert spaces associated with each spectral triple.)  Thus the real-valued linear functions on $\KH$ have bounded commutators with $D_{\oplus}$ and hence, the dense subalgebra condition is satisfied.\\

The operator, $(D_{\oplus}^{2}+I)^{-1}$ is compact, as the set of eigenvalues of $D_{\oplus}$ is the disjoint union of the eigenvalues of $D_{\KH}$ and $\mathfrak{D}_{\KH}$.  Indeed, the union is countable and can be arranged in a non-increasing order according to which the eigenvalues tend to zero.  The self-adjointness of $D_{\oplus}$ is also clearly inherited from its summands.  A function in the image of $\pi_{\oplus}$ is densely defined on $\KH$;  so the representation is faithful.\\

To prove the claim of recovery of Kigami's distance, we need to verify that 

\[||D_{\oplus}f||_{\infty,\KH}=\lip_{g}(f),\] 

\noindent for any $f$ in the domain of $D_{\oplus}$.  Indeed, by \lemref{lemma1} and \lemref{lemma2},

\[||D_{\oplus}f||_{\infty,\KH}=\max\{||D_{\KH}||_{\infty,\KH},||\mathfrak{D}_{\KH}||_{\infty,\KH}\}=\lip_{g}(f).\]

\noindent It then follows immediately that $d_{\oplus}=d_{geo}$.
\end{proof}

\section{Concluding Comments and Future Research Directions}

In this final section, we discuss several possible avenues for future investigation connected with the results of this paper.

\subsection{Spectral dimension and measure vs. Hausdorff dimension and measure}

We conjecture that the spectral dimension $\partial$ of the Dirac operator $D_{\KH}$ (for any of the spectral triples considered in this paper) is equal to the Hausdorff dimension $H$ of ($\KH,d_{geo}$), the harmonic gasket equipped with the harmonic geodesic metric:  $\partial= H$.\\

Moreover, we conjecture that (by analogy with the results obtained in \cite{CrIvLap} for the Euclidean Sierpinski gasket, as well as results and conjectures in \cite{Con1, Con2, ConSull, Lap3, Lap4, KL2}), the harmonic spectral measure, defined as the positive Borel measure naturally associated (via the Dixmier trace) with the given Dirac operator $D=D_{\KH}$, is proportional to the normalized Hausdorff measure $\mathcal{H}=\mathcal{H}_{H}$, defined as the normalized $H$-dimensional Hausdorff measure of the metric space $(\KH, d_{geo})$.  (Recall that by definition, $\mathcal{H}$ is the probability measure naturally associated with the standard $H$-dimensional Hausdorff measure of $(\KH, d_{geo})$.)\\

More specifically, if $Tr_{\omega}$ is any suitable Dixmier trace, then, for every $f\in C(\KH)$, we have (with $\partial=H$, the Hausdorff dimension of         
 $(\KH, d_{geo}))$:

\begin{equation} \label{DixTrace}
Tr_{\omega}\left(D^{-\partial}_{\KH} f\right)=c\int_{\KH} f d\mathcal{H},
\end{equation}

\noindent for some positive constant $c$ (equal to the spectral volume of $(\KH, d_{geo})$, see \cite{KL2, Lap3, Lap4}).  (It follows from the results in the present paper that the left-hand side of Equation \eqref{DixTrace} makes sense;  see, e.g., \cite{Con2, Lap3, Lap4}.) \\

We note that the exact counterpart of the result conjectured just above is obtained in \cite{CrIvLap} in the case of the standard Sierpinski gasket $K$, equipped with the (intrinsic) Euclidean geodesic metric.

\subsection{Global Dirac operator and Kusuoka Laplacian}

We expect that a suitable modification of the constructions provided in this paper should yield a global Dirac operator on $\KH$, and an associated spectral triple, with the same spectral dimension $\partial=H$ and corresponding spectral volume (proportional to the harmonic Hausdorff measure, as in Equation \eqref{DixTrace}), and whose square coincides with (or is in some sense spectrally equivalent to) the Kusuoka Laplacian (that is, minus the Laplacian with respect to the Kusuoka measure).  Some further discussion of this topic will be provided in Subsection 7.4 below.\\

\begin{rmk} \label{rmk6}
Recently, after this work was completed {\normalfont{(}}but independent of it{\normalfont{)}}, generalizing to the specific case of the harmonic gasket and the Kusuoka Laplacian Weyl's asymptotic formula for p.c.f. {\normalfont{(}}i.e., finitely ramified{\normalfont{)}} fractals obtained in {\normalfont{\cite{KL1}}} by Jun Kigami and the first author {\normalfont{(}}see also the later paper {\normalfont{\cite{KL2})}}, and also using results from {\normalfont{\cite{Kig4, Kig5}}}, Naotaka Kajino {\normalfont{(\cite{Kaj1}}} and, especially, {\normalfont{\cite{Kaj2, Kaj3})}} has determined the leading spectral asymptotics of the Kusuoka Laplacian on $\KH$.  In particular, he has shown in {\normalfont{\cite{Kaj1}}} that {\normalfont{(}}twice{\normalfont{)}} the spectral dimension of the Kusuoka Laplacian coincides with the Hausdorff dimension $H$ of $(\KH, d_{geo})$.  Furthermore, in {\normalfont{\cite{Kaj2}}}, he has shown that the Hausdorff measure of $(\KH, d_{geo})$ can be recovered from the leading asymptotics of the Kusuoka Laplacian restricted to an arbitrary {\normalfont{(}}nonempty{\normalfont{)}} open subset of $\KH$ {\normalfont{(}}including, of course, $\KH$ itself{\normalfont{)}}.  These results in {\normalfont{\cite{Kaj2,Kaj3}}} are consistent with the conjectures made in Subsections 7.1 and 7.2 above.  Furthermore, we expect that some of the techniques developed in {\normalfont{\cite{Kaj1, Kaj2, Kaj3}}} will be very useful in addressing and eventually resolving these conjectures, in this and more general settings.

Finally, we note that it is also shown in {\normalfont{\cite{Kaj1, Kaj2, Kaj3}}} that the Hausdorff and Minkowski {\normalfont{(}}i.e., box-counting{\normalfont{)}} dimensions of $(\KH, d_{geo})$ coincide {\normalfont{(}}which is of interest in light of {\normalfont{[31--37]}}, for example{\normalfont{)}}, and that $1<H<2$.
\end{rmk}

\subsection{Energy measure on the gasket}

Based in part on the results of \cite{CrIvLap}, various refinements and extensions of the spectral triples discussed in \cite{CrIvLap} were recently introduced by Erik Christensen, Cristina Ivan, and Elmar Schrohe in \cite{CrIvSc}.  In particular, in \cite{CipGuidIso1}, using the refinements introduced in \cite{CrIvSc}, along with the earlier results and methods of \cite{CrIvLap}, Fabio Cipriani, Daniele Guido, Tommaso Isola and Jean-Luc Sauvageot have shown that the Dirichlet energy form on the Euclidean Sierpinski gasket $K$ can also be recovered from the Dirac operator (and the associated spectral triple) via a suitable Dixmier trace construction.  In light of the results of the present paper and the conjectures made in Subsections 7.1 and 7.2, it is natural to expect that the results of \cite{CipGuidIso1} can be extended to the harmonic gasket (as well as eventually, more general fractals). Namely, conjecturally, not only the Hausdorff dimension and Hausdorff measure of $(\KH, d_{geo})$, but also the energy form on the gasket can be recovered (via a Dixmier trace construction) from a suitable modification of the spectral triples discussed in this paper and in Subsection 7.2 above.  In the process of establishing such a result, it would be helpful to further examine the potential connections between the Dirichlet form, the harmonic geodesic metric on $\KH$, and the effective resistance metric (or intrinsic metric) on $K$, as transported to $\KH$ via the homeomorphism $\Phi$ (see \cite{Kig2, Kig3}).\\

Finally, we note that the modification in \cite{CrIvSc} of the spectral triple constructed in \cite{CrIvLap} may be better suited to the study of the noncommutative topology and $K$-homology of the fractals studied in the present paper, particularly for the harmonic gasket $\KH$ (once our own extended construction has been taken into account).  This question remains to be explored, in conjunction with suitable modifications of the various spectral triples constructed in this paper, including in Section 6. 

\subsection{Geometric analysis on the harmonic gasket}

It would be interesting to further develop geometric analysis on the harmonic Sierpinski gasket $\KH$, viewed as a measurable Riemannian manifold (in the sense of \cite{Kig4, Kig5}).  In the long term, one should be able to extend to this setting the differential calculus on smooth (Riemannian) manifolds, including the notions of differential forms and (metric) connections.  At least for this important special example, this would be a significant step towards realizing aspects of the research program outlined in [32--36].  (The recent results obtained in \cite{CipGuidIso2} for differential 1-forms on the Euclidean gasket may be useful in this setting;  see also \cite{CipSauv} along with the survey article \cite{HiTep} and the relevant references therein.)  Again, in the long term, we expect aspects of geometric analysis to be developed from the present perspective on a broad class of fractal manifolds.

{}








\end{document}